\documentclass[a4paper,10pt,reqno]{amsart}
\usepackage{amscd,amssymb,graphicx,color,epigraph}

\usepackage[alphabetic,initials]{amsrefs}

\usepackage[a4paper, top=1in, bottom=1in, left=1.2in, right=1.2in]{geometry}

\usepackage{amssymb}
\usepackage{amsthm}
\usepackage{calligra}
\usepackage{bm}
\usepackage{latexsym}
\usepackage{amsmath}
\usepackage[dvipsnames]{xcolor}
\usepackage[T1]{fontenc}
\usepackage[colorlinks=true, linkcolor=Bittersweet, citecolor=blue, urlcolor=teal]{hyperref}
\usepackage[capitalize,nameinlink]{cleveref}
\usepackage{mathrsfs}
\usepackage{quiver}
\usepackage{tikz-cd}

\usepackage[all,cmtip]{xy}

\usepackage{listings}
\usepackage[center]{caption}
\theoremstyle{plain}
\usepackage[nodisplayskipstretch]{setspace}
\usepackage{ stmaryrd }

 \numberwithin{equation}{section}

\newtheorem{theorem}{Theorem}[section]
\newtheorem{proposition}[theorem]{Proposition}

\newtheorem{Question}[theorem]{Question}
\newtheorem{lemma}[theorem]{Lemma}

\newtheorem{corollary}[theorem]{Corollary}

\newtheorem{conjecture}[theorem]{Conjecture}

\newtheorem{example}[theorem]{Example}

\newtheorem*{procedure}{Riso-stratification procedure}

\theoremstyle{definition}

\newcommand{\appsection}[1]{\let\oldthesection\thesection
\renewcommand{\thesection}{Appendix \oldthesection}
\section{#1}\let\thesection\oldthesection}

\newtheorem{definition}[theorem]{Definition}

\theoremstyle{remark}

\newtheorem{remark}[theorem]{Remark}

\DeclareMathOperator{\spec}{Spec}

\DeclareMathOperator{\Hom}{Hom}

\DeclareMathOperator{\HomA}{\mathscr{H}\text{\kern -3pt {\calligra\large om}}\,}

\def\Z{{\mathbb{Z}}}

\def\A{{\mathbb{A}}}

\def\O{{\mathcal{O}}}
\def\m{{\mathfrak{m}}}

\DeclareMathOperator{\Spec}{Spec}

\def\arrow#1{\mathop{\longrightarrow}\limits^{#1}}
\makeatletter
\providecommand{\leftsquigarrow}{%
  \mathrel{\mathpalette\reflect@squig\relax}%
}
\newcommand{\reflect@squig}[2]{%
  \reflectbox{$\m@th#1\rightsquigarrow$}%
}
\makeatother

\title{Functorial stratifications
of singularities\\ in characteristic 0}

\author[Vicente Monreal]{Vicente Monreal}
\email{vicente.monreal@hhu.de}
\address{Mathematisches Institut, Heinrich-Heine-Universit\"at D\"usseldorf, Germany.}

\begin{document}

\begin{abstract}
We prove that the riso-stratifications of singularities of affine algebraic sets introduced by Bradley-Williams--Halupczok are embedding independent and can be computed \'etale locally. Building on this, we upgrade their procedure, originally formulated in non-archimedean and model-theoretic language, into a sharp canonical and functorial stratification process for schemes of finite type over fields of characteristic 0. Our work enables comparisons with classical approaches to singularities and makes these stratifications available without requiring familiarity with the original methods.
\end{abstract}

\maketitle

\section{Introduction}

The notion of stratifying a scheme $X$ of finite type over a field refers to a decomposition of $X$ into a disjoint union of locally closed subschemes which are smooth and of pure dimension. Conceptually, a stratification should detect the smooth locus $X^\mathrm{sm}$ and reflect how singular each point is within the singular locus. 

A given object usually admits more than one stratification. Determining a notion of well behaved stratification that allows canonical constructions is a central problem in singularity theory. This is difficult both because there are several candidates for what well behaved should mean (Whitney's conditions \cite{W65}, Verdier's regularity \cite{V76}, compatibility with classical algebraic invariants \cite{Z65}), and because canonical constructions are rarely available, even when the desired properties are fixed. This problem is particularly important in the context of resolution of singularities, where stratifications play a central role in the resolution algorithm introduced by Hironaka over fields of characteristic 0 \cite{Hir64}. Here, stratifications are needed for selecting the smooth centers of the blow-ups defining the algorithm. The canonicity of the stratification is what makes the algorithm itself canonical, and sharpness of the stratification can be converted into both termination and functoriality of the resolution process. The classical solution to this requires, at each step of selecting a blow-up center, a choice of a hypersurface of maximal contact. The canonicity problem raised by these choices was addressed by Villamayor \cite{V89} and Bierstone–Milman \cite{BM97}. Further refinements of functoriality of the process have been obtained by Encinas–Hauser \cite{EH02} and Włodarczyk \cite{W05} in the classical setting, and by Abramovich–Temkin–Włodarczyk \cite{ATW24}, who replaced ordinary blowups by stack-theoretic weighted blow-ups. However, all of these constructions still depend, directly or indirectly, on the notion of maximal contact. This remains the principal obstruction to extending resolutions to positive characteristic \cite{H10}. 

Some approaches to resolutions try to avoid stratifications altogether. The most prominent heuristic is the Nash blow-up proposed by Nash \cite{N95} and Semple \cite{S54}, which is known to resolve curves and surfaces when iterated with normalization \cite{S90}, and some further families in higher dimensions \cite{GT14}. However, this method was recently shown to fail in dimensions four and higher \cites{C25,C26}. The problem of producing canonical stratifications therefore seems unavoidable.

A recent line of work has approached the problem of canonical stratifications in characteristic 0 from a completely different direction. Using non-archimedean and model-theoretic techniques, in particular cell decomposition in valued fields, Bradley-Williams and Halupczok \cite{BWH25} introduced a stratification process for algebraic subsets of a given affine space $\A_k^n$ (or more generally, definable objects in a suitable language, e.g. semialgebraic sets). For a fixed embedded object $X\subset \A_k^n$, their process uniquely determines a stratification of $X$ which they name the ``riso-stratification of $X\subset \A_k^n$''. Riso-stratifications of embedded objects are strictly finer than any minimal Whitney stratification, and capture motivic information not visible to classical regularity conditions. In this paper, we prove that such stratifications are embedding independent and can be computed \'etale locally. From this, we obtain an improved version of them in global contexts:   
\begin{theorem}[Theorems \ref{thm: propiedades de riso-strat}, \ref{thm: riso-strat are funcotiral} and \ref{Prop: Speading properties}]\label{thm: A}
    Let $X$ be a reduced scheme of pure dimension which is of finite type over a field $k$ of characteristic 0. There is a canonical sequence of subschemes of $X$:
    $$\mathcal{S}_\bullet(X/k)=\{\mathcal{S}_d(X/k)\}_{0\leq d\leq \dim X},$$
    determining what we call the riso-stratification of $X$ over $k$. The following properties and regularity conditions are satisfied by this decomposition $\mathcal{S}_\bullet(X/k)$: 
    \begin{enumerate}
        \item[\textbf{(1)}] It is a stratification. Namely, $X=\bigsqcup_{0\leq d\leq \dim X}\mathcal{S}_d(X/k),$ each finite union $\mathcal{S}_{\leq d}(X/k)$ is a closed subscheme of $X$ and each stratum $\mathcal{S}_{d}(X/k)$ is either empty or a smooth subscheme of $X$ of pure dimension $d$. In particular, $\mathcal{S}_{\dim X}(X/k)=X^\mathrm{sm}$.\vspace{0.5ex}
        
         \item[\textbf{(2)}] It satisfies Whitney conditions, in the sense of \cite{BWH25}*{4.6.4}. In particular, multiplicities are constant along connected components of each stratum \cites{T73,T77,T82}.
         \vspace{0.5ex}

        \item[\textbf{(3)}] It captures motivic information of $X$. More precisely, local motivic Poincar\'e series are trivial along each stratum in the sense of \cite{BWH25}*{5.4.4(2)}.
        \vspace{0.5ex}

        \item[\textbf{(4)}] It is compatible with base change along smooth morphisms. That is, for every smooth morphism $f\colon Y\to X$ over $k$ of relative dimension $N$, we have $\mathcal{S}_{r}(Y/k)=\varnothing$ for $r<N$ and: 
        $$ \mathcal{S}_{N+d}(Y/k)=f^{-1}\mathcal{S}_d(X/k),$$
        for every $0\leq d\leq \dim X$. In particular, riso-stratifications over $k$ commute with \'etale localizations.
        \vspace{0.5ex}

        \item[\textbf{(5)}] It is compatible with base field extensions. That is, for every field extension $F/k$ we have:
        $$\mathcal{S}_d(X_F/F)=\mathcal{S}_d(X/k)\times_k\spec(F),$$
        for every $0\leq d\leq \dim X$, where $X_F=X\times_k\spec(F)$ is regarded as a scheme of finite type over $F$.
        \vspace{0.5ex}

        \item[\textbf{(6)}] It has good spreading properties. More precisely, suppose  there is a domain $R$ with $\mathrm{Frac}(R)=k$ and a reduced scheme $\tilde{X}$ of finite presentation over $R$ such that $X\simeq \tilde{X}\times_R\spec(k)$. Then, there are constructible subschemes $\mathcal{S}_d(\tilde{X})\subset \tilde{X}$ uniquely determined by $\tilde{X}$ such that:$$\mathcal{S}_d(X/k)=\mathcal{S}_d(\tilde{X})\times_R\spec(k) ,$$
        for each $0\leq d\leq \dim X$.
    \end{enumerate}
\end{theorem}
We point out that local properties (1) to (3) were already described by Bradley-Williams--Halupczok in their work. Similarly, properties (5) and (6) essentially follow from the definability results already obtained in the embedded context. On the other hand, property (4) and the globalization of the previous are due to our reconstruction of the stratification process.

The local stratification process underlying Theorem \ref{thm: A} is complicated and, for the moment, can only be performed following its original formulation. That is, we need to appeal to some heavy model-theoretic results to even define riso-stratifications. Let us briefly describe this in the embedded situation, leaving a more detailed discussion for the later section \S2. If $X\subset \A_k^n$ is an affine scheme of pure dimension over a field of characteristic 0, its riso-stratification is computed by first considering a base change of the form: 
$$X_K=X\times_k\spec(K) \subset \A_K^n,$$
where $K$ is a suitable valued field extending $k$ (see Definition \ref{def:admissible}). Upstairs, the non-archimedean techniques involving \emph{risometries} (which play the same role as isometries in classical analytic situations) are used to measure the singularities of $X_K$ via the invariants called \emph{riso-triviality spaces}. These invariants are used to decompose $X_K(K)\subset K^n$ into sets whose regularity properties are inductively refined through the \emph{shadow iteration process}. This iterative construction outputs a decomposition $(S_d)_{0\leq d\leq \dim X}$ of $X_K(K)$ into constructible sets (this is highly non-trivial, since every notion involved is formulated in terms of analytic information captured by the valuation of $K$ and the additive structure of the ambient space $\A_K^n$). These constructible sets $S_d$ are proved to descend back to constructible subsets of $X$ over $k$, which define a stratification of $X$ that is independent of the chosen valued field $K$. This output is precisely $\mathcal{S}_\bullet (X/k)$.

Our main results rely on a proof of embedding independence for the riso-stratification process, and this requires, in particular, making sense of the procedure in algebro-geometric terms. For example, the original definition of \emph{riso-triviality spaces} uses a model-theoretic language that obscures their intrinsic nature. We reconstruct them with a perspective that allows us to regard them as invariants of formal germs of singularities. If $k$ is a field of characteristic 0, $X$ is a reduced scheme of finite type over $k$ and $x\in X$ is a closed point, then the riso-triviality space of $X$ at $x$ is a canonical $\kappa(x)$-linear subspace of the tangent space of $X$ at $x$ (see Definition \ref{Def: rtsp of x} and Corollary \ref{Cor: rtdsp independent of K}):
$$ \mathrm{rtsp_x}(X)\subset T_xX.$$
The first step in the riso-stratification process of $X$ is determined by the level sets of riso-triviality dimensions:
$$\mathrm{rtd}_x(X):=\dim_{\kappa(x)}\mathrm{rtsp}_x(X)\in \mathbb{Z}_{\geq 0},$$ as $x$ ranges over closed points of $X$ (it is still not known if these level sets already compute the riso-stratification in algebraic situations. See Definition \ref{def: riso-strat} and Question \ref{Question 1}).

We define riso-triviality spaces in arbitrary characteristic and find the first algebraic constraint for them, expressed as a relation with the tangent cone. This is the main tool for the explicit computation of riso-stratifications of simple objects in Section \S6:

\begin{theorem}[Theorem \ref{Prop: invariant image in the cone} and Corollary \ref{Cor: rtdsp independent of K} ]\label{Thm: B}
    Let $X$ be a scheme of finite type over a field $k$ of arbitrary characteristic, $x\in X$ a closed point. Then, we have:
    $$ C_x(X)(k)+\mathrm{rtsp}_x(X)(k)= C_x(X)(k),$$
    where $C_x(X)$ denotes the tangent cone of $X$ at $x$. In particular, there is a canonical containment $\mathrm{rtsp}_x(X)\subset C_x(X)$ and a dimensional bound $\mathrm{rtd}_x(X)\leq \dim X.$
\end{theorem}
\begin{remark}
    This is the first evidence of riso-triviality spaces being well behaved in characteristic $p>0$. In this context, it is unclear whether the level sets of $\mathrm{rtd}_x(X)$ are constructible, since we cannot appeal to model-theoretic results.  
\end{remark}
\begin{remark}
    This result is conceptually motivated by the results of Garc\'ia-Ramirez about t-stratifications (the model-theoretic predecessors of riso-stratifications), and their relations to tangent cones in more analytic settings \cite{G17}.
\end{remark}

There are two main questions arising naturally from our work. The first one is:
\begin{Question}\label{Question Intro 1}
    Is it possible to describe riso-stratifications without appealing to model-theoretic techniques?
\end{Question}
By this we mean in particular understanding the \emph{shadow iteration process} in more depth (see Question \ref{Question 1} for a more technical phrasing of this). The second, and perhaps more important, question is whether riso-stratifications are compatible with existing algorithms for resolutions of singularities. More precisely:
\begin{Question}\label{Question Intro 2}
    Is it possible to resolve the singularities of a given $k$-scheme $X$ by repeatedly blowing-up regular centers of the form $\mathcal{S}_{\leq d}(X/k)$?
\end{Question}
Of course, just blowing-up the worst stratum of a Whitney stratification is known to fail in general when seeking algorithms of resolutions. However, it is not clear whether those situations can be overcome by blowing-up further finite unions of strata (e.g. Whitney umbrella). Alternatively, one can also use only the worst stratum by replacing ordinary blow-ups with stack-theoretic weighted blow-ups, as explained and performed by McQuillan \cite{McQuillan20}.

 Motivated by our results, we propose a conjecture that aims to give a partial answer to Question \ref{Question Intro 1} and a positive answer to Question \ref{Question Intro 2}. Let us recall some definitions before giving a precise statement:
\begin{definition}
    Let $X$ be a scheme of finite type over a field $k$ and $x\in X$ a closed point. The ridge invariant of $X$ at $x$ is defined by:
    $$\mathrm{Rid}_x(X):=\mathrm{Stab}_{T_xX}\big(C_x(X)\big).$$
Here the stabilizer construction captures non-reduced structure in a delicate manner and can be formalized in terms of representability of a functor, see \cite{BHM10}*{2.1}.
\end{definition}
Stratifying singularities in a way that is compatible with the ridge invariant is one of the key conceptual steps within Hironaka's process of resolution of singularities \cite{Hir70}. Indeed, Giraud's results \cites{Gir72,Gir74,Gi75} characterize the ridge as the tangent cone of a maximal contact variety.
\begin{conjecture}
    Let $X$ be a reduced scheme of finite type over a field $k$ of characteristic 0 and $x\in X$ a closed point. Then, we have a canonical inclusion $$\mathrm{rtsp}_x(X)\subset \mathrm{Rid}_x(X).$$
\end{conjecture}
\begin{remark}
Our Theorem \ref{Thm: B} establishes that $\mathrm{rtsp}_x(X)$ stabilizes 
$C_x(X)_\mathrm{red}$ inside $T_xX$. The conjecture asserts the stronger statement that $\mathrm{rtsp_x}(X)$ also stabilizes the non-reduced structure of $C_x(X)$.
\end{remark}
A proof of this conjecture for normal germs of singularities $x\in X$ will be the subject of forthcoming work.\vspace{1ex}

Our methods are characteristic independent. Therefore, any extension of the results of \cite{BWH25} to positive characteristic would be compatible with our improvements of the riso-stratifying procedure. However, extensions of this kind seem impossible from the model-theoretic side, and so a more promising path to determining whether Theorem \ref{thm: A} extends to positive characteristic is to find an alternative proof that does not appeal to model-theoretic arguments, and this in turn is essentially Question \ref{Question Intro 1}.\vspace{1ex}

The structure of the paper is as follows. In \S2 we fix conventions and put our work into context by giving an informal explanation of the original riso-stratification procedure. In \S3 we develop an algebraic version of the non-archimedean notions of RV-quotients, risometries and riso-triviality dimension invariants. In \S4 we establish some properties of riso-triviality dimensions: They are intrinsic, depend only on formal germs of singularities, and are obstructed by the tangent cone via Theorem \ref{Thm: B}. In \S5 we reconcile our setup with the original by showing that our notions of riso-triviality agree in a suitable sense. We then deduce embedding independence and smooth functoriality of the embedded riso-stratification process from the properties satisfied by our constructions, and culminate in a proof of Theorem \ref{thm: A}. In \S6 we compute explicit examples of riso-stratifications and raise two technical open questions.

\subsection*{Acknowledgments} We thank Hugo Zock, Otto Overkamp, Sebastián Muñoz-Thon, Roberto Villaflor and B\'arbara Ram\'irez for valuable discussions while preparing this document.
We thank Joaqu\'in Moraga, Mark Spivakovsky, Stefan Schröer and Kay Rülling for their feedback which helped improve the exposition of the paper. Special thanks to Immanuel Halupczok, who introduced the author to riso-stratifications and is currently supervising his PhD thesis project. This work was financially supported by the Deutsche Forschungsgemeinschaft, since the author is a member of the research training group GRK 2240: Algebro-Geometric Methods
in Algebra, Arithmetic and Topology.

\tableofcontents

\section{Preliminaries on Riso-Stratifications, Setup and Conventions}
 
The techniques underlying the riso-stratification process were developed systematically over the last decade \cites{BWH25,CH25,G20,Hal10,Hal11,Hal14,Hal14b,HalYin18}. In this paper we reconstruct a modified version of their geometric notions and invariants. Before doing so it is worth having in hand a conceptual picture of \emph{what} the riso-stratification of a scheme is and \emph{how} it is built, avoiding the model-theoretic formalism on which this was originally formulated. The purpose of this section is to provide such a picture, fix our conventions, and make precise the obstructions that arise when trying to translate this model-theoretic machinery to algebro-geometric situations. We refer to \cite{Hal23} for a more elementary introduction, closer to the original heuristic and containing further instructive examples.

\subsection{Risometries: A non-Archimedean Analog of Isometries}
Let $K$ be a valued field of characteristic 0 that is algebraically closed and spherically complete\footnote{We will only use the fact that spherically complete fields have a complete valuative topology, in the sense that every Cauchy sequence has a unique limit; see \cite{BWH25}*{2.2.6} for a formal definition.}, with valuation $v\colon K \to \Gamma \cup \{\infty\}$, valuation ring $\mathcal{O}_K$, and residue field $\kappa$ which is also a subfield of $K$ via a section of the residue map. We use additive notation for $v$.
 
The valuation endows the affine space $K^n$ with a non-archimedean notion of distance: The \emph{valuative distance} between two points $x,y \in K^n$ is the valuation $v(x-y)$ of their difference, where $v(x_1,\dots,x_n) := \min_i v(x_i)$. The closed ball of radius $\gamma \in \Gamma$ centered at $x$ is:
\[
  B_{\ge \gamma}(x) = \{\, x' \in K^n \mid v(x'-x) \ge \gamma \,\}.
\]
Taking these balls as a basis turns $K^n$ into a topological space. Because the valuative distance defines an ultrametric, every element of a ball is a center, every ball is at once open and closed, and the resulting topology on $K^n$ is totally disconnected. This is a genuine complication when one wishes to analyze an embedded algebraic set $X \subset K^n$, since the topological and analytic tools that are so effective over $\mathbb{R}$ or $\mathbb{C}$ cannot be replaced directly.
 
The standard non-archimedean approach to deal with this lack of structure is to improve the space itself. Rather than work with the bare set of points under a single valuation, it is possible to enlarge the input by admitting many valuations, or multiplicative seminorms, at once and equipping the geometric result with a structure sheaf, passing to the analytic spaces of Berkovich or the adic spaces of Huber.
 
The riso-approach proceeds along an entirely different line which allows the study of more general definable objects. Instead of enriching the space in question to restore good topological properties, it leaves $K^n$ untouched and seeks a suitable replacement for the analytic notion of \emph{isometry}, requiring compatibility with the model-theoretic techniques available over valued fields such as cell decomposition or its generalized version of Hensel minimality.
 
A first guess for a replacement would be to keep the analytic definition untouched, calling a map isometry in case it preserves valuative distances of the form $v(x-y)$. This is too coarse: This condition encodes order of vanishing of differences of points but not the direction along which two points separate. A way to fix this is to remember, alongside the valuation of differences, leading terms. This enriched datum is encoded in what we call the $\mathrm{RV}$-quotient of $K^n$, and produces the following genuine analog of isometries:

\begin{definition}[informal; cf.\ Definitions \ref{Def: Risometries} and \ref{def: old risometry}]
A \emph{risometry} is a bijection $\varphi$ between subsets of $K^n$ that preserves the $\mathrm{RV}$-class of every difference. That is, for all $x,y$ in its domain, $\varphi(x)-\varphi(y)$ has the same valuation \emph{and} the same leading term as $x-y$, equivalently
\[
  v\big(\varphi(x)-\varphi(y)-(x-y)\big) > v(x-y).
\]
\end{definition}
 
 The terminology is purely a choice of name: A risometry is an \emph{$\mathrm{RV}$-isometry}. This is the source of the prefix \emph{riso-} which is present in almost every notion built from this equivalence, namely riso-triviality, the riso-stratification, and the riso-tree of \cite{BWH25}. Translations are risometries, and risometries are stable under composition and inversion. This entire story is about how to work \emph{up to risometry}.
\subsection{Riso-triviality: Degree of Smoothness up to Risometries}
 
Our purpose is to measure singularities of algebraic subsets of $K^n$ by non-archimedean means. This requires us to first fix a local notion of smoothness. The geometric insight to be exploited is that, infinitesimally, a smooth point of an algebraic set looks like a linear subspace, and is in particular translation-invariant along as many directions as its dimension. This suggests measuring the local degree of smoothness of an algebraic set $Z\subset K^n$ at a point in relation to the additive structure of the ambient space, namely by the number of directions along which $Z(K)$ is translation-invariant on small valuative ball.

\begin{definition}[informal; cf. Definition \ref{def: translation invariant}]
Let $W \subset \kappa^n$ be a $\kappa$-linear subspace. A subset $A\subset K^n$ is \emph{$W$-translation-invariant on a valuative ball $B\subset K^n$} if $(A + W(K))\cap B = A$.
    
\end{definition}
 Demanding local translation-invariance to measure smoothness is, however, far too rigid. Even smooth algebraic subsets of $K^n$ are, in general, locally invariant under no translation. We therefore relax the definition by asking for local translation-invariance only \emph{up to risometry}.
 
\begin{definition}[informal; cf.\ Definitions \ref{def: riso-triviality} and \ref{def: old riso-triviality}]
A subset $A$ of a valuative ball $B\subset K^n$ is \emph{$W$-riso-trivial on $B$} if there is a risometry $\varphi\colon B \to B$ for which $\varphi(A)$ is $W$-translation-invariant on $B$. The \emph{riso-triviality dimension} of $A$ on $B$ is
\[
  \mathrm{rtd}_B(A) := \max\{\, \dim_\kappa W \mid A \text{ is $W$-riso-trivial on } B \,\}.
\]
\end{definition}
 
This notion can already be used to provide invariants measuring the singularities of geometric objects. Let $Z \subset \kappa^n$ be an algebraic set of pure dimension and $z \in Z(\kappa)$ a closed point. Restricting the set $Z(K)$ to the infinitesimal ball $B_{\geq 0}(z)\subset K^n$ and computing its riso-triviality dimension associates to the point $z$ a non-negative integer:
\[
  \mathrm{rtd}_z(Z) \in \mathbb{Z}_{\ge 0},
\]
the local degree of smoothness of $Z$ at $z$, measured up to risometry. It is a non-trivial fact that this attains its maximal value $\mathrm{rtd}_x(Z) = \dim Z$ at smooth points (Corollary \ref{cor: smooth implies riso-trivial}). However, it is still not known if this invariant alone detects algebraic smoothness.
 
\subsection{The Embedded Riso-Stratification Procedure}
 
Sorting the $\kappa$-points of $Z$ by the value of the riso-triviality dimension invariant produces a first partition called the \emph{first shadow of $Z$}. For each $0\leq d\leq \dim Z$ this first shadow is determining the set:
\[
  \mathrm{Sh}_d\big(Z\big)(\kappa) = \{\, z \mid \mathrm{rtd}_z(Z) = d \,\}\subset Z(\kappa).
\]

A first surprising result that comes from the model theoretic formalisms is constructibility of these $\kappa$-point sets. The first shadow of $Z$ is a good approximation to a stratification of its singularities, but not yet a stratification in the strict sense. The further obstruction is the \emph{locally closed} requirement on the strata. There is no immediate reason for the partial unions $\bigsqcup_{e \le d}\mathrm{Sh}_e$ to be Zariski closed. This is expected to be true for algebraic sets (see question \ref{Question 1}), but does not hold in the full generality of definable objects such as semi-algebraic ones, see \cite{BWH25}*{4.2.4}.
 
This first shadow construction can be refined by an inductive procedure, the \emph{iterated shadow}. In order to formalize this, one extends the definition of riso-triviality dimensions to handle tuples of sets. Then, by knowing that each stratum of a first shadow is constructible, it is possible to again decompose $Z(\kappa)$ using the level sets of the following riso-triviality dimensions:
\[\mathrm{rtd}_{B_{\geq 0}(z)}\big(Z(K),\mathrm{Sh}_0(Z)(K),\dots, \mathrm{Sh}_{\dim Z}(Z)(K)\big)\in \mathbb{Z}_{\geq 0},\]
as $z$ ranges in $Z(\kappa)$. This process can be iterated appealing to the model theoretic results that assert constructible strata at each step, see Definition \ref{def: shadow iteration}. Another non-trivial fact that comes from the model-theoretic world is that this iterative process stabilizes, and the resulting decomposition of $Z(\kappa)$ is precisely what we call the \emph{riso-stratification of $Z\subset K^n$}.
 
This already explains conceptually the riso-stratification procedure when working under a fixed embedding $Z\subset K^n$, for particularly well behaved valued fields $K$. This is rarely a geometric situation of interest. Nonetheless, the techniques can be extended to more general contexts. Let us describe more precisely the general stratifying procedure:
\begin{definition}\label{def:admissible}
A field extension $K/\kappa$ is \emph{admissible over a domain $R$} if:
\begin{itemize}
  \item $K$ is an algebraically closed, spherically complete valued field whose valuation ring $\mathcal{O}_K$, with maximal ideal $\mathfrak{m}_{\mathcal{O}_K}$, satisfies $\mathcal{O}_K/\mathfrak{m}_{\mathcal{O}_K} \simeq \kappa$, inducing a section of the residue map.
  
  \item The algebraically closed field $\kappa$ extends $\mathrm{Frac}(R)$.
\end{itemize}
\end{definition}
 
\begin{procedure} Let $R$ be a domain of characteristic 0 and $I\subset R[x_1,\dots,x_n]$ a finitely generated radical ideal corresponding to a subscheme $Z\subset \A^n_R$ of pure relative dimension over $R$. The riso-stratification of $Z$ with respect to this embedding, in the sense of \cite{BWH25}*{4}, is obtained by applying the following conceptual steps: ~\vspace{1ex}
\begin{enumerate}
    \item Choose an admissible field extension $K/\kappa$ over $R$. \vspace{0.5ex}
    
    \item Base change to obtain the embedding $Z_\kappa\subset \A_\kappa^n$. Perform the shadow iteration construction for the set $Z_\kappa(K)\subset K^n$. This returns uniquely determined constructible subsets $S_d\subset \A_\kappa^n$ stratifying $Z_\kappa$ with further regularity conditions.\vspace{0.5ex}

    \item Prove that the stratification $\{S_d\}_{0\leq d\leq \dim Z_\kappa}$ of $Z_\kappa$ is the base change of a unique stratification of $Z$ which does not depend on the choice made in step (1).
\end{enumerate}
The resulting  constructible sets $\{S_d\}_{0\leq d\leq \dim Z}$ are what we call the \emph{riso-stratification of }$Z\subset R^n$. When $R$ is a field, this is indeed a stratification satisfying further regularity conditions. 
\end{procedure}

This general procedure is based upon Hensel minimality and cell decomposition in valued fields, which are techniques previously developed in \cites{CH11,CHal14,CH20}. Some important remarks, that stress the relevance of these model-theoretic tools, is that Step (2) works without information about $R$. Similarly, step (3) depends only on the field $\kappa$ and not on the choice of a valued field $K$ for which $K/\kappa$ is admissible over $R$. 
 
\subsection{Drawbacks of the Original Formulation}
Despite having results that establish riso-stratifications as meaningful constructions for embedded algebraic sets, in the sense of compatibility with their intrinsic structure (Whitney conditions, motivic information), it is unclear in \cite{BWH25} whether these stratifications are embedding independent. In fact, the model theoretic language and geometric setup they use is incompatible with the mere question, which only makes sense after restricting our attention to algebraic sets rather than definable objects over more general languages. When trying to clarify this situation one encounters the following obstacles:

\begin{itemize}
    \item The RV-relation used to define risometry classes depends on a choice of origin. Even when working up to translations, one can check that this quotient is not compatible with localizations at $K$-points of $Z$ which do not factor through the \'etale stalk of a geometric point $x\in Z(\kappa)$.\vspace{0.5ex}
    
    \item The definition of risometry and riso-triviality depend on the additive structure of an ambient space, even after fixing the RV-quotient. As it was remarked by Halupczok \cite{Hal23}*{3.4.4}, this quotient does not retain the additive structure. This is to say, working up-to-risometries is sensitive to the dimension of the ambient space.\vspace{0.5ex}

    \item Any modified intrinsic construction of riso-triviality dimension invariants used to run step (2) of the riso-stratification procedure must lose information: The original construction recovers in particular the codimension of the embedding.
\end{itemize}
In what follows, we will deal with all these obstacles by redefining riso-triviality dimensions of subsets of $Z_\kappa(\O_K)$. That is, we will only use the restricted information recovered by the inclusion:
$$(\O_K)^n\subset K^n.$$
This has enough geometric content to see valuative balls $B\subset K^n$ of radius $\geq 0$ centered at $\O_K$-points, which will be sufficient for recovering the riso-stratification procedure. With this change of perspective, it is possible to define the valuative topology and RV-quotient in a more intrinsic way. This will be used to prove that riso-triviality dimension invariants are intrinsic to formal germs of singularities. 

\subsection*{Conventions}
 For a ring $R$, the term $R$-scheme will mean a reduced scheme which is of finite presentation over $R$ and of pure relative dimension. For $K/\kappa$ an admissible extension over $R$ and a $\kappa$-scheme $X$, we use the term $\O_K$-points of $X$ when referring to the set $X(\O_K):=\Hom_{k}(\spec \O_K,X).$

\section{Riso-Triviality Spaces and Dimensions}
In this section, we fix a field extension $K/\kappa$ which is admissible over a domain $R$ (of arbitrary characteristic) and develop a modified version of the riso-triviality dimension invariants introduced in \cite{BWH25} for studying algebraic subsets of $\A_\kappa^n(K)$. We use additive notation for the valuation $v\colon K\to \Gamma \cup \{\infty\}$.

\subsection{The Valuative Topology and its RV-structure}
Let $Z$ be a $\kappa$-scheme. The following definitions endow the set of $\O_K$-points of $Z$ with a valuative structure. 

\begin{definition}
    We say that $\alpha\in Z(\O_K)$ is an arc on $Z$. The \emph{base of $\alpha$} is the closed point of $Z$ defined by composition with the canonical morphism $\spec \kappa\to \spec \O_K$.  
\end{definition}

\begin{definition}
    For every $\gamma\in \Gamma$ consider the ideal $I_{>\gamma}:=\{f\in \O_K~|~v(f)> \gamma \}\subset \O_K$ and its associated quotient $\O_K^{\leq \gamma}=\O_K/I_{>\gamma}$. The \emph{base of} $\eta\in Z(\O_K^{\leq \gamma})$ is the $\kappa$-point of $Z$ defined by the image of the closed point of $\spec( \O_K^{\leq\gamma})$.
\end{definition}

\begin{definition}
     Let $\gamma\in\Gamma$. The \emph{$\gamma$-th truncation map}
     $\theta_\gamma\colon Z(\O_K)\longrightarrow Z\big(\O_K^{\leq\gamma}\big)$
     is the one induced by the quotient map $\rho_\gamma\colon \O_K\to \O_K^{\leq\gamma}$. These maps commute, and do not alter base points. More formally, for every $\gamma_1>\gamma_2\in\Gamma$ we have $I_{>\gamma_1}\subset I_{>\gamma_2}$ and the further ring quotient $\rho_{\gamma_2,\gamma_1}\colon \O_K^{\leq\gamma_2}\to \O_K^{\leq\gamma_1}$ induces a map  $\theta_{\gamma_1,\gamma_2}$ that makes the following diagram commute:
     \[\begin{tikzcd}
	{Z(\O_K)} & {Z\big(\O_K^{\leq\gamma_1}\big)} \\
	& {Z\big(\O_K^{\leq\gamma_2}\big)}
	\arrow["{\theta_{\gamma_1}}", from=1-1, to=1-2]
	\arrow["{\theta_{\gamma_2}}"', from=1-1, to=2-2]
	\arrow["{\theta_{\gamma_1,\gamma_2}}", from=1-2, to=2-2]
\end{tikzcd}\]
\end{definition}
\begin{remark}
    Truncation maps are very often not surjective over singular base points.   
\end{remark}
\begin{definition}
    Let $z\subset Z(\O_K)$ and $\gamma\in\Gamma$. The \emph{closed valuative ball on $Z$ of radius $\gamma$ centered at $z$} is 
    $$B_{Z,z}(\gamma):=\theta_\gamma^{-1}\circ\theta_\gamma(z).$$ As a convention, we write $B_{Z,z}=B_{Z,z}(0)$ for the ball consisting of all $\O_K$-points of $Z$ based at $z$. We say that $A\subset Z(\O_K)$ is \emph{infinitesimal} if it is composed of arcs having a same base point.
\end{definition}
Since $\O_K$-points factor through formal neighborhood of their base points, and such local arcs can be lifted uniquely \cite{Artin}, we have a canonical correspondence: 
$$B_{Z,z}\cong \Hom_{\text{loc }\kappa\text{-alg}}\big(\widehat{\O_{Z,z}},\O_K\big).$$
In this way, elements of the free $\kappa$-module $\kappa[B_{Z,z}]$ can be regarded as $\kappa$-linear functions in $\mathrm{Fun}_{\kappa}\big(\widehat{\O_{Z,z}},\O_K\big)$. We use this to extend the valuation of $\O_K$ to a function:
\begin{align*}
    v\colon \kappa[B_{Z,z}]&\longrightarrow \Gamma\\
    \rho&\longmapsto \min_{l\in \mathfrak{m}_z} v\big(\rho(l)\big),
\end{align*} where $\mathfrak{m}_{Z,z}$ denotes the associated finitely generated maximal ideal.
\begin{definition}\label{Def: Val-Distance}
    Let $\alpha,\beta\in Z(\O_K)$ be two arcs based at $z\in Z(\kappa)$. The \emph{valuative distance between $\alpha$ and $\beta$} is:
    $$v(\alpha,\beta):=v(\alpha-\beta)=\min_{l\in\m_{Z,z}}v\big(\alpha(l)-\beta(l)\big).$$
    We also define the valuative order of $\alpha$ by $v(\alpha):=v(\alpha,z)$.
\end{definition}
The assumption of $K$ being spherically complete says that every Cauchy sequence with respect to the valuative distance on $\O_K$-points of $Z$ has a unique limit. This notion will play a role via the following definition:
\begin{definition}
    A set $A\subset Z(\O_K)$ is \emph{valuatively closed} if it is closed under sequential limits. In particular, every closed valuative ball $B\subset Z(\O_K)$ is indeed valuatively closed.
\end{definition}
The objects which we will analyze during the following sections are the following:
\begin{definition}
    A \emph{$\O_K$-tuple of $Z$} is a finite tuple of sets of the form:
    $$ \mathcal{A}=\big(A_1,\dots,A_\rho\big),$$
    such that $A_i\subset Z(\O_K)$ for each $1\leq i\leq \rho$ and its support $\mathrm{supp}(\mathcal{A}):=\bigcup_{i}A_i$ is valuatively closed. We say that $\mathcal{A}$ is trivial if $\mathrm{supp}(\mathcal{A})=\varnothing$. If $f\colon C \to C' $ is a set theoretic function such that $\mathrm{supp}(\mathcal{A})\subset C$ and $C'\subset Z(\O_K)$, then we denote by $f(\mathcal{A})$ the $\O_K$-tuple of $Z$ consisting of elements $f(A)$ with $A\in\mathcal{A}$. For every $C\subset Z(\O_K)$ we denote by $\mathcal{A}|_C$ the $\O_K$-tuple composed of elements of the form $A\cap C$ with $A\in\mathcal{A}$.
\end{definition}

\begin{definition}\label{def: RV-relation}
    For a fixed $\kappa$-scheme $Z$ we define the \emph{rv-relation for arcs $\alpha,\beta\in Z(\O_K)$} as follows:
    $$\alpha\sim_{\mathrm{rv}} \beta \iff \underset{\displaystyle \big(v(\alpha,\beta)>v(\alpha)\big)~\vee~ (\alpha= \beta ).}{\text{ Both arcs have the same base point and}}$$
    We denote $\mathrm{RV}_{\O_K}(Z):=Z(\O_K)/\sim_{\mathrm{rv}}$, and the quotient map by $\mathrm{rv}_Z\colon Z(\O_K)\to \mathrm{RV}_{\O_K}(Z)$. This makes sense because of the following:
\end{definition}
\begin{lemma}\label{Lemma: RV-is-equivalence}
    The relation $\sim_\mathrm{rv}$ on $Z(\O_K)$ is an equivalence relation.
\end{lemma}
\begin{proof}
    Reflexivity follows from the definition. Let $\alpha,\beta\in Z(\O_K)$ be two different arcs, suppose $\alpha\sim_\mathrm{rv}\beta$ and let $z\in Z(k)$ be the common base point. Consider the finitely generated maximal ideal $\mathfrak{m}_z\subset \O_{Z,z}$ and take $l\in \mathfrak{m}_z$ be such that $v(\alpha(l))=v(\alpha)$. The strong triangle inequality implies
    $v\big(\alpha(l)-\beta(l)\big)\geq \min\big\{v\big(\alpha(l)\big),v\big(\beta(l)\big)\big\}$ and equality cannot hold because of the rv-condition. Thus, $v(\alpha(l))= v(\beta(l))$ which implies $ v(\alpha)\geq v(\beta)$ and we have symmetry. For transitivity, the same argument says that $\alpha\sim_\mathrm{rv}\beta$ and $\beta\sim_\mathrm{rv}\varepsilon$ implies $v(\alpha)=v(\beta)=v(\varepsilon)$ and the result follows.
\end{proof}
\begin{example}\label{ex: RV-quot and affine embeddings}
  Consider $\A_\kappa^n:=\Spec(\kappa[x_1,\dots,x_n])$. Once we have chosen coordinates, we get a set theoretic identification that describes the additive group structure of $\A_\kappa^n(\O_K)$:
  $$\A_\kappa^n(\O_K)\simeq (\O_K)^n.$$ 
  This can be used to check that the rv-relation recovers directional information with respect to base points. For example, for a pair of arcs $\alpha=(r_1,\dots,r_n)$ and $\beta=(r'_1,\dots,r'_n)$ based at the origin, we have:
    $$ \alpha\sim_\mathrm{rv}\beta \iff r_i \equiv r'_i \mod{I_{v(\alpha)}}~\forall i.$$
\end{example}
Consider an embedding of $\kappa$-schemes $Z_1\hookrightarrow Z_2$. The induced set-theoretic injection on arcs is compatible with truncation maps, i.e. for every $\gamma\in\Gamma$ we have a commutative diagram:
\[\begin{tikzcd}
	{Z_1(\O_K)} & {Z_2(\O_K)} \\
	{Z_1\big(\O_K^{\leq\gamma}\big)} & {Z_2\big(\O_K^{\leq\gamma}\big)}
	\arrow[hook, from=1-1, to=1-2]
	\arrow["{\theta_\gamma}"', from=1-1, to=2-1]
	\arrow["{\theta_\gamma}", from=1-2, to=2-2]
	\arrow[hook, from=2-1, to=2-2]
\end{tikzcd}\]
This implies preservation of valuative distances, in particular each valuative ball of $Z_1$ is the restriction of a valuative ball of $Z_2$. We will repeatedly make use of this concept, and so we fix some notation for it:
\begin{definition}
    Let $Z_1\hookrightarrow Z_2$ be an embedding of $\kappa$-schemes and $B\subset Z_1(\O_K)$ a valuative ball. The \emph{restriction of $B$ to $Z_2$} is the valuative ball $B|_{Z_2}=B\cap Z_2(\O_K)$ of $Z_2$, we also say that $B$ is a lift of the valuative ball $B|_{Z_2}$.
\end{definition}
Embeddings are also compatible with RV-quotients. i.e. the following diagram commutes:
\[\begin{tikzcd}
	{Z_1(\O_K)} & {Z_2(\O_K)} \\
	{\mathrm{RV}_{\O_K}(Z_1)} & {\mathrm{RV}_{\O_K}(Z_2)}
	\arrow[hook, from=1-1, to=1-2]
	\arrow["{\mathrm{rv}_{Z_1}}"', from=1-1, to=2-1]
	\arrow["{\mathrm{rv}_{Z_2}}", from=1-2, to=2-2]
	\arrow[hook, from=2-1, to=2-2]
\end{tikzcd}\]

This notion is particularly interesting when applied to affine embeddings $Z\hookrightarrow\A^n_\kappa$, since it concludes that part of the directional information featured by $Z(\O_K)$ in the ambient space is already contained in $\mathrm{RV}_{\O_K}(Z)$. 

\begin{example}
    Suppose $\kappa\llbracket t\rrbracket\hookrightarrow\O_K$ in a compatible way with the valuation of $\O_K$ and consider $Z=\spec(\kappa[x,y]/(y-x^2))$. This induces $Z(\kappa\llbracket t\rrbracket)\subset Z(\O_K)$ and we have a set theoretic identification: $$Z(\O_K)\simeq \{\big(a,b\big)\in (\O_K)^2~|~ a^2=b\}.$$
    In this way, the valuative order of an arc in $Z(\kappa\llbracket t\rrbracket)$ based at the origin is just the $t$-adic order of its $x$-coordinate. Thus, for every pair of arcs $(a,b),(c,d)\in B_{Z,0_{\A^2}}\cap Z(\kappa\llbracket t\rrbracket)$ we have:\begin{align*}
        (a,b)\sim_\mathrm{rv}(c,d)&\iff a\text{ and }c\text{ have the same }t\text{-adic order and leading $\kappa$-coefficient.}
    \end{align*}
    
\end{example}

\subsection{Risometries and Riso-Triviality}
In this section we define the central notion of risometries, which allows the construction of riso-triviality dimension invariants associated to $\O_K$-tuples of $\A_\kappa^n$. For this purpose, we fix a presentation $\A_\kappa^n=\spec(\kappa[x_1,\dots,x_n])$ and consequently a group scheme isomorphism $\A_\kappa^n\simeq\mathbb{G}_a^n$.
 \begin{definition}\label{def: translation invariant}
    Let $\mathcal{A}$ be a $\O_K$-tuple of $\A_\kappa^n$, $B$ a valuative ball of $\A_\kappa^n$ and $W\subset \A_\kappa^n$ a $\kappa$-linear subspace. We say $\mathcal{A}$ is \emph{$W$-translation-invariant on $B$} if $\mathcal{A}|_B$ is a non-trivial $\O_K$-tuple such that:
    $$A+B_{W,0}(\gamma)=A ,$$
    for every $A\in \mathcal{A}|_B$, where $\gamma\in\Gamma$ is the radius of $B$.
\end{definition}

\begin{definition}\label{Def: Risometries}
    Let $A,A'\subset \A_\kappa^n(\O_K)$ be infinitesimal subsets. A set-theoretic bijection $\varphi\colon A\to A'$ is a \emph{risometry}, if for every $a,b\in A$:
    $$\mathrm{rv}\big(\varphi(a)-\varphi(b)\big)=\mathrm{rv}(a-b).$$
    This comparison makes sense since both $a-b$ and $\varphi(a)-\varphi(b)$ are arcs based at the origin of $\A_\kappa^n$. 
\end{definition}
\begin{remark}
    Risometries preserve valuative distances. Translations are risometries. Compositions and inverses of risometries are risometries. A function with infinitesimal domain that preserves rv-classes of differences is automatically a risometry to its image. 
\end{remark}
\begin{definition}\label{def: riso-triviality}
Let $\mathcal{A}$ be a $\O_K$-tuple of $\A_\kappa^n$ and $B$ a valuative ball. We say that $\mathcal{A}$ is \emph{$W$-riso-trivial on $B$}, if there exists a risometry $\varphi\colon \mathrm{supp}(\mathcal{A}|_B)\to C$, with $C\subset B$, such that $\varphi(\mathcal{A}|_B)$ is $W$-translation-invariant on $B$. A risometry $\varphi$ satisfying this condition will be called a \emph{$W$-straightener of $\mathcal{A}$ on $B$}. The \emph{riso-triviality dimension of $\mathcal{A}$ on $B$} is:
$$\mathrm{rtd}_{B}(\mathcal{A}):= \max_{\underset{\mathcal{A}\textit{ is $V$-riso-trivial}}{V\subset \A_\kappa^n}} \dim_\kappa V .$$
We say that $A\subset \A_\kappa^n(\O_K)$ is $W$-riso-trivial on $B$ when the singleton $\{A\}$ is.
\end{definition}
\begin{remark}
    The riso-triviality dimension of a $\O_K$-tuple $\mathcal{A}$ does not depend on the group isomorphism $\A_\kappa^n(\O_K)\simeq (\O_K)^n$.
\end{remark}
\begin{remark}
    A $\O_K$-tuple $\mathcal{A}$ being $W$-riso-trivial on a valuative ball $B\subset\A_\kappa^n(\O_K)$ is a strictly stronger condition than each one of its elements being $W$-riso-trivial on $B$.
\end{remark}
\begin{example}
    Consider the parabola $Z=\spec (\kappa[x,y]/(y-x^2))$, the $\kappa$-linear subspace $W=\spec(\kappa[y])\subset \A_\kappa^2$ and the valutive ball $B=B_{\A_\kappa^2,0}$. We claim that $Z(\O_K)$ is $W$-riso-trivial on $B$: The linear projection to $W$ induces a set-theoretic map $\pi_W\colon Z(\O_K)\to W(\O_K)$. This restricts to a bijection:
    \begin{align*}
        \varphi\colon B_{Z,0}&\longrightarrow B_{W,0}\\
        \alpha&\longmapsto \pi_W(\alpha).
    \end{align*}
    It is easy to check that this bijection is a risometry: From the identification $\A^n_{\kappa}(\O_K)\simeq (\O_K)^2$ each arc $\alpha\in B_{Z,0}$ can be regarded as a tuple of the form $\alpha=(\alpha_1,\alpha_1^2)$ such that $v(\alpha_1)>0$. In this way, for every pair of arcs $\alpha,\beta\in B_{Z,z}$ we have:
    \begin{align*}
        v\big(\varphi(\alpha)-\varphi(\beta)-\alpha+\beta\big)&=v\big( (\alpha_1,0)-(\beta_1,0)-(\alpha_1,\alpha_1^2)+(\beta_1,\beta_1^2)\big)\\
        &=v\big(\alpha_1^2-\beta_1^2\big)\\
        &>v(\alpha-\beta).
    \end{align*}
    This is to say, $\varphi(\alpha)-\varphi(\beta)\sim_\mathrm{rv}\alpha-\beta$ as arcs of $B_{\A_\kappa^2,0}$ and so $\varphi$ is a riometry.  
\end{example}
We will very frequently restrict our attention to $\O_K$-tuples $\mathcal{A}$ whose support is defined on $\kappa$-linear subspaces of $\A_\kappa^n$. In this setting, it is important to detect that the riso-triviality dimension invariant is compatible with restriction to subspaces: 
\begin{lemma}\label{almost-preserve-directions}
Let $V,W\subset\A_\kappa^n$ be $\kappa$-linear subspaces and $\mathcal{A}$ a $\O_K$-tuple of $V(\O_K)\subset \A_\kappa^n(\O_K)$. If $\mathcal{A}$ is $W$-riso-trivial on a valuative ball $B$ of $\A_\kappa^n$, then $W\subset V$. In particular, $\mathcal{A}$ is also $W$-riso-trivial on the valuative ball $B|_V$ of $V$ and we have $\mathrm{rtd}_{B}(\mathcal{A})=\mathrm{rtd}_{B|_V}(\mathcal{A})$.
\end{lemma}
\begin{proof}
By contradiction, suppose $\mathcal{A}$ is $W$-riso-trivial on a valuative ball $B$ of $\A_\kappa^n$ and $W\not\subset V$. We can reduce the analysis to the case $W\cap V=\{0_{\A_\kappa^n}\}$, $\mathcal{A}|_B=\mathcal{A}$ and $\A_\kappa^n\simeq V\oplus W$ where both are proper subspaces. Let $\varphi\colon \mathrm{supp}(\mathcal{A})\to C$ be a $W$-straightener of $\mathcal{A}$ on $B$, and let $A\in\mathcal{A}$. Up to translations both infinitesimal sets $A,\varphi(A)$ contain the fixed point $0_{\A_\kappa^n}$. Set linear coordinates $l_1,\dots,l_n$ to $\A_\kappa^n$ such that $W$ is defined by the vanishing of the first $r$ of them. For $a\in A$ such that $\varphi(a)\in W$ is non-trivial, then:
\begin{align*}
    v(a)&=\min_{1\leq i\leq r}v\big(l_i(a) \big)\\
    &=\min_{1\leq i\leq r}v\big(l_i\big(a-\varphi(a)\big) \big)\\
    &\geq v\big(a-\varphi(a)\big).
\end{align*}
This contradicts the risometry condition $v\big(a-\varphi(a)\big)>v(a)$.
\end{proof}

\subsection{Adapted Risometries and Riso-Triviality Spaces}

One of the most important properties of riso-triviality is that it is additive, in terms of its compatibility with operations on $\kappa$-linear subspaces of $\A_\kappa^n$. This notion is the underlying reason of existence of riso-triviality spaces associated to $\O_K$-tuples, which geometrize the construction of the riso-triviality dimension invariant. We will make sense of this via developing adapted risometries, which will also be relevant when describing algebraic obstructions to riso-triviality later on.

 We will recurrently use some translation functions for the following proofs, let us fix the notation: For $\alpha\in\A_\kappa^n(\O_K)$ define the translation by $\alpha$ as the function:\begin{align*}
      \tau_\alpha\colon \A_\kappa^n(\O_K)&\longrightarrow\A_\kappa^n(\O_K)\\
      \beta&\longmapsto \beta+\alpha~.
  \end{align*}

\begin{lemma}\label{Lemma: Improved-Risometries}
    If a $\O_K$-tuple $\mathcal{A}$ of $\A_\kappa^n$ is $W$-riso-trivial on a valuative ball $B\subset \A_\kappa^n(\O_K)$, then there exists a $W$-straightener $\varphi\colon \mathrm{supp}(\mathcal{A}|_B)\to C$, such that the following diagram commutes:
    \[\begin{tikzcd}
	\mathrm{supp}(\mathcal{A}|_B) & C \\
	{W(\O_K)}
	\arrow["\varphi", from=1-1, to=1-2]
	\arrow["{\pi_W}"', from=1-1, to=2-1]
	\arrow["{\pi_W}", from=1-2, to=2-1]
\end{tikzcd}\]
    Where $\pi_W$ is any linear projection to the subspace $W\subset \A_\kappa^n$. If a risometry $\varphi$ satisfies this condition, then we say it is a $W$-straightener of $\mathcal{A}$ on $B$ adapted to $W$.
\end{lemma}
\begin{proof}
    Assume $\mathcal{A}|_B=\mathcal{A}$ and consider $\varphi_0\colon \mathrm{supp}(\mathcal{A})\to C'$ a $W$-straightener of $\mathcal{A}$ on the valuative ball $B$. Up to translations, we can assume $0\in \mathrm{supp}(\mathcal{A})$ and that it is a fixed point. We claim that the map $\tau\colon C'\to \A_\kappa^n(\O_K)$ defined by $\varphi_0(\alpha)\mapsto \tau_{\pi_W(\alpha-\varphi_0(\alpha))}\big(\varphi_0(\alpha)\big)$ is a risometry to its image $C$. Indeed, set $l_1,\dots,l_n$ linear coordinates of $\A_\kappa^n$ such that $W$ is defined by the vanishing of the first $r$ of them. The images under $\varphi_0$ of any pair of distinct arcs $\alpha,\beta\in A$ have a same base point, moreover:
    \begin{align*}
        v\big(\tau\big(\varphi_0(\alpha)\big)-\tau\big(\varphi_0(\beta)\big)-\varphi_0(\alpha)+\varphi_0(\beta)\big)&=v\big(\pi_W\circ\varphi_0(\alpha)-\pi_W\circ\varphi_0(\beta)-\pi_W(\alpha)+\pi_W(\beta)\big)\\
        &=\min_{i>r}v\big( l_i\big(\varphi_0(\alpha)\big)-l_i\big(\varphi_0(\beta)\big)-l_i(\alpha)+l_i(\beta)\big)\\
        &\geq v\big(\varphi_0(\alpha)-\varphi_0(\beta)-\alpha+\beta \big)\\
        &>v(\alpha-\beta).
    \end{align*}
     By construction $l_i(\alpha)=l_i\big(\tau\circ\varphi_0(\alpha)\big)$ for $i>r$. Therefore, the result follows from considering $\tau\circ\varphi_0$, which inherits the property of being a $W$-straightener of $\mathcal{A}$ on $B$.
\end{proof}

\begin{proposition}\label{Prop: Improved riso}
    Let $\mathcal{A}$ be a $\O_K$-tuple of $\A_\kappa^n$, $W\subset \A_\kappa^n$ a $\kappa$-linear subspace and $\pi_W\colon \A_\kappa^n\to W$ any linear projection. Then, $\mathcal{A}$ is $W$-riso-trivial on a valuative ball $B\subset\A_\kappa^n(\O_K)$ if and only if for every $w\in \pi_W(B)$ there exists a $W$-straightener of $\mathcal{A}$ adapted to $W$ on $B$ of the form:
    $$\varphi_w\colon \mathrm{supp}(\mathcal{A}|_B)\longrightarrow \big(\pi_W|_{\mathrm{supp}(\mathcal{A}|_B)}\big)^{-1}(w)+ B_{W,w_0}(\gamma),$$
    where $\gamma\in\Gamma$ is the radius of $B$. In particular, if $\mathcal{A}$ is $W$-riso-trivial on $B_{\A_\kappa^n,0}$, then we can assume that the adapted risometry $\varphi_{0}$ also preserves RV-quotients, i.e. the following diagram commutes:
   \[\begin{tikzcd}
	\mathrm{supp}(\mathcal{A}|_{B_{\A_\kappa^n,0}}) & {\big(\pi_W|_{\mathrm{supp}(\mathcal{A}|_{B_{\A_\kappa^n,0}})}\big)^{-1}(0)+ B_{W,0}} \\
	{\mathrm{RV}_{\O_K}(\A_\kappa^n)}
	\arrow["{\varphi_{0}}", from=1-1, to=1-2]
	\arrow["{\mathrm{rv}}"', from=1-1, to=2-1]
	\arrow["{\mathrm{rv}}", from=1-2, to=2-1]
\end{tikzcd}\]    
\end{proposition}
\begin{proof}
    Assume $\mathcal{A}|_B=\mathcal{A}$. For the non-trivial implication, suppose $\mathcal{A}$ is $W$-riso-trivial on $B$ and let $\varphi\colon \mathrm{supp}(\mathcal{A})\to C$ be $W$-straightener of $\mathcal{A}$ adapted to $W$ on $B$, in the sense of Lemma \ref{Lemma: Improved-Risometries}.
    Up to translations we can assume that $\mathrm{supp}(A)$ contains the origin and that it is a fixed point. In that case, $\pi_W\big(\mathrm{supp}(\mathcal{A})\big)=B_{W,0}(\gamma)$ where $\gamma\in\Gamma$ is the radius of $B$. Consider for each $w\in B_{W,0}(\gamma)$ the risometry $\tilde{\varphi}_w:=\varphi^{-1}\circ\tau_w\circ\varphi\colon \mathrm{supp}(\mathcal{A})\to \mathrm{supp}(\mathcal{A})$. It follows from the construction that the family of risometries constructed this way satisfies the following properties:
    \begin{itemize}
        \item[(i)] For every $w_1,w_2\in B_{W,0}(\gamma)$, $\tilde\varphi_{w_1}\circ\tilde\varphi_{w_2}=\tilde\varphi_{w_1+w_2}$.

        \item[(ii)] For every $w\in B_{W,0}(\gamma)$ and $\alpha\in \mathrm{supp}(\mathcal{A})$, $\pi_W\circ \tilde\varphi_w(\alpha)=\pi_W(\alpha)+w$.
        
        \item[(iii)] For all $\alpha\in \mathrm{supp}(\mathcal{A})$ and $w\in B_{W,0}(\gamma)$, $v\left(\tilde\varphi_{w}(\alpha)-\tau_w(\alpha)\right)>v\left(w\right)$.
        
        \item[(iv)] For every $A\in\mathcal{A}$ and $w\in B_{W,0}(\gamma)$, $\tilde{\varphi}_w(A)=A$.
    \end{itemize}
    Fix $w\in B_{W,0}(\gamma)$ and define $\tilde\varphi\colon \mathrm{supp}(\mathcal{A})\to (\pi_W|_{\mathrm{supp}(\mathcal{A})})^{-1}(w)$ by $\alpha\mapsto \tilde\varphi_{\pi_W(\alpha)-w}(\alpha)$, we claim that the injective map $\varphi_w\colon \mathrm{supp}(\mathcal{A})\to B_{\A_\kappa^n,0}$ given by $\alpha\mapsto \tau_{w-\pi_W(\alpha)}\circ\tilde\varphi(\alpha)$ is a risometry to its image, which is precisely $\big(\pi_W|_{\mathrm{supp}(\mathcal{A})}\big)^{-1}(w)+ B_{W,0}(\gamma)$. Set $l_1,\dots, l_n$ linear coordinates of $\A_\kappa^n$ such that $W$ is defined by the vanishing of the first $r$ of them. The construction of $\varphi_w$ and property (ii) yield $l_i(\alpha)=l_i\big(\varphi_w(\alpha)\big)$ for every $i>r$ and $\alpha\in B$. For a pair of distinct arcs $\alpha_1,\alpha_2\in A$ this implies:
    \begin{align*}
        v\big(\varphi_w(\alpha_1)-\varphi_w(\alpha_2)-(\alpha_1-\alpha_2)\big)&=\min_{1\leq i\leq r}v\big(\underbrace{l_i\big(\varphi_w(\alpha_1)-\varphi_w(\alpha_2)-(\alpha_1-\alpha_2)\big)}_{L_i}\big)
    \end{align*} 
    Define $\tilde{\alpha_1}=\tilde\varphi_{\pi_W(\alpha_2-\alpha_1)}(\alpha_1)$ which lives in the same fiber of $\pi_W$ as $\alpha_2$ by (ii). Property (i) also implies $\tilde\varphi(\alpha_1)=\tilde\varphi(\tilde{\alpha_1})$ and we can write:
    \begin{align*}
         v(L_i)=& v\big(l_i\big(\tilde\varphi(\tilde{\alpha_1})-\tilde\varphi(\alpha_2)-(\tilde{\alpha_1}-\alpha_2)+\triangle\big)\big)\\
        \geq& \min\left\{ v\big(l_i\big(\tilde\varphi(\tilde{\alpha_1})-\tilde\varphi(\alpha_2)-(\tilde{\alpha_1}-\alpha_2)\big),v\big(l_i(\triangle)\big) \right\},
    \end{align*}
    where $\triangle=\tilde{\alpha_1}-\alpha_1$. When $1\leq i\leq r$, we get $v\big(l_i(\triangle)\big)>v(\alpha_1,\tilde{\alpha_1})\geq v(\alpha_1,\alpha_2)$ by (iii). Also, $\tilde\varphi$ being a risometry when restricted to the fiber of $\alpha_2$ implies:
    \begin{align*}
        v\big(\tilde\varphi(\tilde{\alpha_1})-\tilde\varphi(\alpha_2)-(\tilde{\alpha_1}-\alpha_2)\big)>v(\tilde{\alpha_1}-\alpha_2)\geq v(\alpha_1,\alpha_2).
    \end{align*}
    This together with the previous equalities verifies the risometry condition. It follows from property (iv) that each $\varphi_w$ is indeed a $W$-straightener of $\mathcal{A}$ on $B$, concluding the main result. For the particular case $B=B_{\A_\kappa^n,0}$, the RV-quotient preserving property of the associated $\varphi_0$ follows from the risometry condition and the fact we can assume $\varphi_0(0)=0$ up to translations.
\end{proof}
\begin{corollary}\label{Cor: riso-triv en familias triviales}
    Let $V,V^\perp\subset \A_\kappa^n$ be $\kappa$-linear subspaces with a given group scheme isomorphism $\A_\kappa^n\simeq V\oplus V^\perp$. Let $\mathcal{A}$ be a $\O_K$-tuple of $V(\O_K)$, and consider the $\O_K$-tuple $\mathcal{A} \times V^\perp$ of $\A_\kappa^n(\O_K)$ defined component-wise via $A+V^\perp(\O_K)\subset \A_\kappa^n(\O_K)$ for $A\in\mathcal{A}$. Then, $\mathcal{A} \times V^\perp$ is $(V^\perp\oplus W)$-riso-trivial on a ball $B\subset \A_\kappa^n(\O_K)$ with $B\cap V(\O_K)\neq \varnothing$ if and only if $\mathcal{A}$ is $W$-riso-trivial on $B|_V$. In particular:$$\mathrm{rtd}_{B}(\mathcal{A}\times V^\perp)=\mathrm{rtd}_{B|_V}(\mathcal{A})+\dim_\kappa V^\perp.$$
\end{corollary}
\begin{proof}
    For the first implication, note that the isomorphism induces a projection $\pi_{V^\perp}\colon \A_\kappa^n\to V^\perp$ for which each component of $\mathcal{A}\times V^\perp$ is a trivial family. This observation can be used to lift a $W$-straightener of $\mathcal{A}$ on $B|_V $ to a $(W\oplus V^\perp)$-straightener of $\mathcal{A} \times V^\perp$ on $B$. For the second implication, suppose $\mathcal{A} \times V^\perp$ is $(W\oplus V^\perp)$-riso-trivial on $B$, assume that $\mathcal{A}|_B=\mathcal{A}$ and consider the linear projection $\pi_{V^\perp}\colon \A_\kappa^n\to V^\perp$ determined by the isomorphism. Let $\pi_{W\oplus V^\perp}\colon \A_\kappa^n\to W\oplus V^\perp$ be a $\kappa$-linear projection, By Proposition \ref{Prop: Improved riso}, there is a $(W\oplus V^\perp)$-straightener $\varphi\colon \mathrm{supp}\big(\mathcal{A}\times V^\perp\big)\cap B\to C$ on $B$, adapted to the projection $\pi_{W\oplus V^\perp}$ for which $C$ is a trivial family. Denote $F=(\pi_{V^\perp}|_C)^{-1}(0)$ and regard it as a subset of $V(\O_K)$. If we denote $i\colon \mathrm{supp}(\mathcal{A})\to \mathrm{supp}\big(\mathcal{A} \times V^\perp\big)$ the canonical inclusion given by the isomorphism, then $\varphi\circ \mathrm{Im}(i)=F$. Indeed, this comes from the fact that $W\subset V$, which implies $\pi_{V^\perp}\circ\pi_{W\oplus V}=\pi_{V^\perp}$ and so $\varphi$ also commutes with the projection to $V^\perp$. By construction $F$ and every image $\varphi\circ i(A)$, with $A\in\mathcal{A}$, are $W$-translation-invariant on $B|_V$ which concludes the claimed result.
    \end{proof}

\begin{proposition}\label{Lemma: riso-triviality es aditiva}
    Let $\mathcal{A}$ be a $\O_K$-tuple of $\A_\kappa^n$ such that $\mathrm{supp}(\mathcal{A})$ is closed in the valuative topology, consider $V,W\subset \A_\kappa^n$ two $\kappa$-linear subspaces and a valuative ball $B\subset\A_\kappa^n(\O_K)$. If $\mathcal{A}$ is $V$-riso-trivial and $W$-riso-trivial on $B$ , then it is also $(V\oplus W)$-riso-trivial on $B$.
\end{proposition}
\begin{proof}
    We can assume without loss of generality that $V\cap W={0_{\A_\kappa^n}}$ and $\mathcal{A}|_B=\mathcal{A}$. Choose a group scheme isomorphism $\A_\kappa^n\simeq V\oplus V^\perp$, so that $W\subset V^\perp$. By Proposition \ref{Prop: Improved riso}, we can assume that $\mathrm{supp}(\mathcal{A})$ and all their elements define trivial families with respect to the projection $\pi_{V}$ determined by the isomorphism, with central fiber $F=(\pi_V|_{\mathrm{supp}(\mathcal{A})})^{-1}(0)\subset V^\perp$. By the same Proposition, there exists a $W$-straightener of $\mathcal{A}$ on $B$ of the form $\psi\colon \mathrm{supp}(\mathcal{A})\to C$ adapted to $W$. Moreover, up to translations we have $0\in \mathrm{supp}(\mathcal{A})$ and $\pi_W\big(\mathrm{supp}(\mathcal{A})\big)=B_{W,0}(\gamma)$ for $\gamma\in \Gamma$ the radius of $B$. Now, for every $w\in B_{W,0}(\gamma)$ denote $F_w=(\pi_W|_F)^{-1}(w)$ and consider the map:
    \begin{align*}
        \tilde{\varphi}_w\colon F_0&\longrightarrow F_w\\
        \alpha&\longmapsto \pi_{V^\perp}\circ \psi^{-1}\circ \tau_w\circ\psi(\alpha)
    \end{align*} 
    The image of this map is indeed $F_w$, but this is a non-trivial check: Let $\beta\in F_w$ and consider $H=\beta+B_{V,0}(\gamma)$. Of course $H$ is closed in the valuative topology, and so its image via a risometry $H'=\psi^{-1}\circ\tau_{-w}\circ\psi(H)\subset (\pi_W|_{\mathrm{supp}(\mathcal{A})})^{-1}(0)$ also is. Since $\mathrm{supp}(\mathcal{A})$ is closed, $F_0$ also is, and by spherical completeness we can choose $\alpha\in F_0$ and $\tilde\alpha\in H'$ such that:
    $$v(\alpha,\tilde{\alpha})=\max_{u\in F_0,v\in H'}v(u,v).$$
    We claim that $\alpha=\tilde\alpha$: Of course $\pi_{V^\perp}(\tilde{\alpha})=\alpha$, and so $\alpha-\tilde\alpha\in B_{V,0}(\gamma)$. Then $\mu=\phi^{-1}\circ\tau_{-w}\circ\phi\circ\tau_{\alpha-\tilde\alpha}\circ\psi^{-1}\circ\tau_w\circ\psi(\tilde\alpha)\in H'$ is such that:
    $$\mathrm{rv}(\mu-\tilde\alpha)=\mathrm{rv}(\alpha-\tilde\alpha),$$
    which implies $v(\alpha-\mu)>v(\alpha-\tilde{\alpha})$ a contradiction. In other words, we have $\tilde{\varphi}_w(\alpha)=\beta$ and so $\tilde{\varphi}_w(F_0)=F_w$. Now, we verify that $\tilde{\varphi}_w$ is a risometry: Set $l_1,\dots, l_n$ linear coordinates of $\A_\kappa^n$ such that $V$ is defined by the vanishing of the first $r$ of them, for $\alpha,\beta\in F_0$ we have:
    \begin{align*}
        v\big(\tilde{\varphi}_w(\alpha)-\tilde{\varphi}_w(\beta)-\alpha+\beta\big)&=v\big(\pi_{V^\perp}\circ \psi^{-1}\circ \tau_w\circ\psi(\alpha)-\pi_{V^\perp}\circ \psi^{-1}\circ \tau_w\circ\psi(\beta)-\alpha+\beta\big)\\
        &=\min_{i\leq r}v\left(l_i\big(\psi^{-1}\circ \tau_w\circ\psi(\alpha)-\psi^{-1}\circ \tau_w\circ\psi(\beta)-\alpha+\beta\big)\right)\\
        &>v(\alpha-\beta)
    \end{align*}
    where we use the fact that $\alpha,\beta$ and their images under $\tilde{\varphi}_w$ live in $V^\perp$, together with the risometry condition satisfied by the composition $\psi^{-1}\circ \tau_w\circ\psi$. Recall that this rv-condition implies injectivity, so now we can guarantee the existence of the risometry $\tilde{\varphi}_w^{-1}\colon F_w\to F_0$ for each $w\in B_{W,0}(\gamma)$. Consider then the map $\varphi\colon F\to F_0\times B_{W,0}(\gamma)$ defined by $\alpha\mapsto \tau_{\pi_W(\alpha)}\circ\tilde{\varphi}^{-1}_{\pi_W(\alpha)}(\alpha)$. We claim that $\varphi$ is a risometry: First, note that for every $\alpha,\beta\in F_0$ and $u,w\in B_{W,0}(\gamma)$ we have
    \begin{align*}
        v\big(\psi(\alpha)-\psi(\beta)-(\alpha-\beta)\big)&>v(\alpha-\beta)\\
        &\geq\min\{v(\alpha-\beta),v(u-v)\}\\
        &=v(\alpha-\beta+u-w)
    \end{align*}
    where the last equality holds since $\alpha-\beta\in W^\perp$ and $u-w\in W$, and so we cannot have $\mathrm{rv}(\alpha-\beta)=\mathrm{rv}(u-v)$ unless both differences are trivial, in which case the condition holds. In other words, we have:
    \begin{align*}
        \mathrm{rv}\big(\psi^{-1}\circ\tau_u\circ\psi(\alpha)-\psi^{-1}\circ\tau_w\circ\psi(\beta)\big)&=\mathrm{rv}\big(\tau_u\circ\psi(\alpha)-\tau_w\circ\psi(\beta)\big)\\
        &=\mathrm{rv}(\alpha-\beta+u-w)
    \end{align*}
    where we also used the risometry condition on $\psi^{-1}$. Similarly, we have: $$\mathrm{rv}\big(\pi_{V}(\psi^{-1}\circ\tau_u\circ\psi(\alpha)-\psi^{-1}\circ\tau_w\circ\psi(\beta))\big)\neq \mathrm{rv}(\alpha-\beta+u-v),$$
    unless both vectors are trivial where the result hods. Indeed, we are comparing the class of an element of $V$ with the class of an element of $V^\perp$. Therefore, a similar computation as before yields:
    $$ \mathrm{rv}\big(\tilde{\varphi}_u(\alpha)-\tilde{\varphi}_w(\beta)\big)=\mathrm{rv}\big(\alpha-\beta+u-w+\pi_V\big(\psi^{-1}\circ\tau_u\circ\psi(\alpha)-\psi^{-1}\circ\tau_w\circ\psi(\beta)\big)\big).$$
     Note that from our previous identities we can deduce:
    $$v\big(\pi_V\big(\psi^{-1}\circ\tau_u\circ\psi(\alpha)-\psi^{-1}\circ\tau_w\circ\psi(\beta)\big)\big)>v(\alpha-\beta+u-w),$$
    This holds since the not-projected representative on the left has the same rv-class as the latter which is in $V^\perp$, so its minimal valuation is not attained at the $V$-coordinates. This valuation inequality implies the following:
    $$\mathrm{rv}(\alpha-\beta+u-v)=\mathrm{rv}\big(\alpha-\beta+u-w+\pi_V\big(\psi^{-1}\circ\tau_u\circ\psi(\alpha)-\psi^{-1}\circ\tau_w\circ\psi(\beta)\big)\big) ,$$
    which together with the above equalities of rv-classes yields:
    $$\mathrm{rv}\big(\tilde{\varphi}_u(\alpha)-\tilde{\varphi}_w(\beta)\big)=\mathrm{rv}\big(\varphi\circ\tilde{\varphi}_u(\alpha)-\varphi\circ\tilde{\varphi}_w(\beta)\big).$$
    This verifies that $\varphi$ is a risometry. Consider the $\O_K$-tuple $\mathcal{F}$ in $V^\perp$ defined component-wise via $F=(\pi_V|_{A})^{-1}(0)$ for $A\in\mathcal{A}$, we claim that $\varphi$ is a $W$-straightener for $\mathcal{F}$ on $B|_{V^\perp}$: Indeed, this just follows from the property of $\psi$ being a $W$-straightener for $\mathcal{A}$ on $B$ adapted to $W$, and the nature of the construction of $\varphi$. We conclude the Proposition by observing that there is an expression of the form $\mathcal{A}=\mathcal{F}\times B_{V,0}(\gamma)$, where $\gamma$ is the radius of $B$, and this can be used via Proposition \ref{Cor: riso-triv en familias triviales} to conclude $(V\oplus W)$-riso-triviality of $\mathcal{A}$ on $B$.  
\end{proof}

\begin{definition}
    Let $\mathcal{A}$ be a $\O_K$-tuple of $\A_\kappa^n$. The \emph{riso-triviality space of $\mathcal{A}$ with respect to a valuative ball $B\subset \A_\kappa^n(\O_K)$}, denoted by $\mathrm{rtsp_{B}(\mathcal{A})}$, is the unique $\kappa$-linear subspace $V\subset \A_\kappa^n$ such that $\mathcal{A}$ is $V$-riso-trivial on $B$ and $\mathrm{rtd}_{B}(A)=\dim V$. This is well defined as a consequence of Proposition \ref{Lemma: riso-triviality es aditiva}.
\end{definition}

\section{On The Algebraicity of Riso-Triviality}

Consider an admissible extension $K/\kappa$ over a domain $R$ of arbitrary characteristic. Let $X$ be a $\kappa$-scheme and $\mathcal{A}$ a $\O_K$-tuple of $X$. In this section, we prove that the value $\mathrm{rtd}_{B}(\mathcal{A})$, where $B\subset \A_\kappa^n(\O_K)$ is a valuative ball with respect to a fixed embedding, is intrinsic to $X$ and depends only on $B|_{X(\O_K)}$ as a non-embedded valuative ball of $X$. We proceed with this idea, in a way that unlocks connections between riso-triviality dimensions and classical algebraic invariants of singularities.

 \subsection{Tangent Spaces and Riso-Triviality Intrinsicness}  The riso-triviality information of $\O_K$-tuples when restricted to $B_{X,x}$ is, in a suitable sense of the word, intrinsic and encoded entirely by the formal germ of the singularity that defines $x\in X(\kappa)$. We elaborate on this:

\begin{definition}
    The geometric Zariski tangent space of $X$ at a closed point $x$ is:
$$T_xX:=\spec\left( \text{Sym}_\kappa\big(\mathfrak{m}_{X,x}/\mathfrak{m}_{X,x}^2\big)\right).$$
\end{definition}
For the given closed embedding $X\hookrightarrow \A_\kappa^n$, every coordinate choice on $\A_\kappa^n$ comes with a geometric presentation $T_xX\hookrightarrow\mathbb{A}_\kappa^n$
that exhibits the tangent space of $X$ at $x$, as a subscheme of $\A_\kappa^n$ containing $x$. Up to a translation, we can assume that $T_xX$ is a $\kappa$-linear subspace and that the coordinates are adapted to it. This choice comes with an associated splitting $s\colon \mathfrak{m}_{X,x}/\mathfrak{m}_{X,x}^2\to \mathfrak{m}_{X,x}$ and a projection $\A_\kappa^n\to T_xX$ which also induces a set theoretic map: $$\pi_s\colon {X(\O_K)}\to T_x{X(\O_K)}.$$
 It follows from Cohen's structure theorem that this map is injective when restricted to $B_{X,x}$. More precisely, the splitting  induces a closed immersion of formal schemes:
$$\pi_s\colon\mathrm{Spf}\left(\widehat{\O_{X,x}}\right) \hookrightarrow\mathrm{Spf}\left(\widehat{\O_{T_xX,x}}\right),$$
which in turn induces a injection from $\O_K$-points of $X$ based at $x$, to $\O_K$-points of $T_xX$ based at $x$. The following proposition both recovers this injectivity and detects further structure on the restricted map:

\begin{proposition}
    \label{Prop: Projection is risometry}
    The map $\pi_s\colon B_{X,x}\to T_x{X(\O_K)}$ restricted to its image, is a risometry.
\end{proposition}

\begin{proof}
    Set linear coordinates $l_1,\dots l_n$ of $\A_\kappa^n$ such that $T_xX$ is defined by the vanishing of the first $r$ of them and $\pi_s\colon \A_\kappa^n\to T_xX$ corresponds to the inclusion of $\kappa[l_{r+1},\dots,l_n]$ in $\kappa[l_1,\dots,l_n]$. Denote by $\O_{X}$ the coordinate ring of $X$, whose maximal ideal corresponding to the origin is generated by the images of $l_{r+1},\dots, l_n$. Consider a pair of different arcs $\alpha,\beta\in B_{X,x}$, and notice that:
    $$\min_{1\le i\le n}v\big(l_i(\alpha-\beta)\big)=\min_{r<i} v\big(l_i(\alpha-\beta)\big) <\min_{i\le r}v\big(l_i(\alpha-\beta)\big). $$
    Indeed, this comes from the fact that the images $\overline{l_i}$ with $i\le r$ in $\widehat{\O_{X,x}}$ lie in $\mathfrak{m}_{X,x}^2$, and so their valuation on $\alpha-\beta$ is strictly greater than  $v(\alpha,\beta)$, which has to be detected by the images $\overline{l_j}$ with $j>r$. This is to say, we have the following relation:
    
    \begin{align*}
        v\big(\pi_s(\alpha-\beta)-\alpha+\beta\big)&=\min_{i\le r} v\big(l_i(\alpha-\beta)\big)\\
        &>v(\alpha-\beta).
    \end{align*}
    This is exactly the risometry condition for $\pi_s$ when restricted to its image.
\end{proof}

\begin{corollary}\label{cor: smooth implies riso-trivial}
    Consider a closed embedding $X\hookrightarrow \A_\kappa^n$ and $x\in X(\kappa)$ defining a smooth point. Then $\mathrm{rtd}_{B}(B|_X)=\dim X$ for every valuative ball $B$ of $\A_\kappa^n$ based at $x$ such that $\mathcal{A}|_B$ is non-trivial.
\end{corollary}
\begin{proof}
This is just a consequence of Proposition \ref{Prop: Projection is risometry} and the fact that any projection $\pi_{s}:B_{X,x}\to B_{T_xX,x}$ is surjective by formal smoothness.    
\end{proof}

\begin{proposition}\label{Prop: Correspondence of rtspaces}
    Consider a closed embedding $X\hookrightarrow \A_\kappa^n$ and $x\in X(\kappa)$. Fix a coordinates of $\A_\kappa^n$ determining a geometric presentation $T_xX\hookrightarrow \A_\kappa^n$, and choose a section $s\colon\m_{X,x}/\m_{X,x}^2\to \m_{X,x}$ with associated map $\pi_s\colon B_{X,x}\to T_x{X(\O_K)}$. Then, we have for every $\O_K$-tuple $\mathcal{A}$ of $X$ a dimension preserving correspondence: 

    $$\left\{\begin{array}{l}
        \underset{k\text{-linear subsp.}}{ W\subset\A_\kappa^n}~\Big|\text{\scriptsize$\mathcal{A}|_B$ is $W$-riso-trivial on $B$}
    \end{array}\right\}\stackrel{1-1}{\longleftrightarrow}\left\{\begin{array}{l}
        \underset{k\text{-linear subsp.}}{\tilde{W}\subset T_xX}~\Big|\text{\scriptsize$\pi_s(\mathcal{A}|_B)$ is $\tilde{W}$-riso-trivial on $\pi_s(B)$}
    \end{array}\right\}$$
    for every valuative ball $B$ of $\A_\kappa^n$ based at $x$ such that $\mathcal{A}|_B$ is non-trivial. In particular, the riso-triviality space $\mathrm{rtsp}_{B}(\mathcal{A})\subset \A_\kappa^n$ sits canonically inside $T_xX$ up to translating the latter to the origin and we have:
    $$\mathrm{rtd}_{B}(\mathcal{A})=\mathrm{rtd}_{\pi_s(B)}\big(\pi_s(\mathcal{A})\big).$$
    This says that the dimension $\mathrm{rtd}_{B}(\mathcal{A})$ is intrinsic to $X$ and depends only on $B|_X$. 
\end{proposition}
\begin{proof}
    We may assume $x=0_{\A_\kappa^n}$, and so we can choose coordinates of $\A_\kappa^n$ and a projection to the realized tangent space $T_xX\hookrightarrow\A_\kappa^n$ in such a way that it agrees with the morphism $\pi_s\colon X\to T_xX$ defined by $s$. By Proposition \ref{Prop: Projection is risometry} and the fact that inverses of risometries are risometries, $\mathcal{A}$ is $W$-riso-trivial on $B$ if and only if $\pi_s(\mathcal{A})$ is on a valuative ball $\tilde{B}$ of $V$ of the same radius centered at a point in $\pi_s(\mathrm{supp}(\mathcal{A}|_B))$. When $\mathcal{A}|_B$ is non-trivial, we can take $\tilde{B}$ to be $\pi_s(B)$. Furthermore, Lemma \ref{almost-preserve-directions} guarantees $W\subset T_xX$ in the cases where $W$-riso-triviality of $\mathcal{A}$ is satisfied and the claimed correspondence is provided by $s$. We also obtain that $\mathrm{rtd}_{B}(\mathcal{A})$ is intrinsic to $X$, since this amounts to the fact that two different embeddings can be compared through $T_xX$ via a same chosen section $s$.
\end{proof}
Now it is possible to extend our main definitions to the non-embedded, global context. First, we have to observe that every projection to linear subspace $\pi_V\colon \A_\kappa^n\to V$ not only sends valuative balls of $\A_\kappa^n$ to valuative balls of $V$, but it does so in a surjective way. Similarly, when $B_{X,x}\subset V(\O_K)$ every valuative ball of $X$ based at $x\in X(\kappa)$ can be regarded as a valuative ball of $V$ restricted to $X(\O_K)$. This can be used to make sense of the following:

\begin{definition}
    Let $\mathcal{A}$ be a $\O_K$-tuple of a $\kappa$-scheme $X$, $x\in X(\kappa)$ and $B\subset B_{X,x}$ a valuative ball. The \emph{riso-triviality space of $\mathcal{A}$ on $B$} is the canonical $\kappa$-linear subspace:
    $$ \mathrm{rtsp}_B(\mathcal{A})\subset T_xX, $$ 
    corresponding to $\mathrm{rtsp}_{\tilde{B}}(\mathcal{A})$ via Proposition \ref{Prop: Correspondence of rtspaces} among all affine opens $U\subset X$ through which $x$ factors, $U\hookrightarrow \A_\kappa^n$ closed embeddings featuring $x$ as the origin, linear projections $\pi_s\colon \A_\kappa^n\to T_xX$ and valuative balls $\tilde{B}\subset \A_\kappa^n(\O_K)$ such that $\pi_s(\tilde{B})|_X=B$ and $\mathcal{A}|_{\tilde{B}}$ is non-trivial. The riso-triviality dimension of $\mathcal{A}$ is $\mathrm{rtd}_B(\mathcal{A})=\dim_\kappa \mathrm{rtsp}_B(\mathcal{A})$.
\end{definition}
The following Proposition says that our non-embedded notion of riso-triviality dimensions can be recovered from the formal germs of the singularities defined by points $x\in X(\kappa)$.

\begin{corollary}\label{Cor: rtd invariante formal} Let $X,Y$ be two $\kappa$-schemes. Suppose there are points $x\in X(\kappa)$ and $y\in Y(\kappa)$ with a given isomorphism: $$\Psi\colon \widehat{\O_{X,x}}\stackrel{\simeq}{\longrightarrow}\widehat{\O_{Y,y}}.$$
Then, the induced bijection $\psi\colon B_{X,x}\to B_{Y,y}$ preserves riso-triviality dimensions. That is, for every $\O_K$-tuple $\mathcal{A}$ of $X$ such that $\mathrm{supp}(\mathcal{A})\subset B_{X,x}$ and valuative ball $B\subset B_{X,x}$ we have: $$ \mathrm{rtd}_B(\mathcal{A})=\mathrm{rtd}_{\psi(B)}\big(\psi(\mathcal{A})\big).$$
\end{corollary}
\begin{proof}
Recall that $\O_K$-points of $X$ centered at $x$ factor through the formal neighborhood of $x$ in $X$, and all $\O_K$-points of this formal neighborhood lift; the same happens for $Y$, therefore the induced bijection $\psi\colon B_{X,x}\to B_{Y,y}$ exists. The isomorphism also induces an isomorphism $\tilde{\Psi}\colon T_xX\to T_yY$, and furthermore allow us to associate to every splitting $s\colon \mathfrak{m}_{{X},x}/\mathfrak{m}_{{X},x}^2\to \mathfrak{m}_{{X},x}$ a splitting $\tilde{s}\colon \mathfrak{m}_{{Y},y}/\mathfrak{m}_{{Y},y}^2\to \mathfrak{m}_{{Y},y}$ in such a way that the projections $\pi_X\colon B_{X,x}\to T_x{X(\O_K)}$ and $\pi_Y\colon B_{{Y},y}\to T_yY(\O_K)$ make the following diagram commute:
\[\begin{tikzcd}
	{B_{X,x}} & {\widehat{(X)}_{x}(\O_K)} & {\widehat{(Y)}_y(\O_K)} & {B_{Y,y}} \\
	{T_x{X(\O_K)}} &&& {T_yY(\O_K)}
	\arrow["\cong"{description}, draw=none, from=1-1, to=1-2]
	\arrow["\psi", curve={height=-25pt}, from=1-1, to=1-4]
	\arrow["{\pi_s}"', from=1-1, to=2-1]
	\arrow["\Psi", from=1-2, to=1-3]
	\arrow["\cong"{description}, draw=none, from=1-3, to=1-4]
	\arrow["{\pi_{\tilde{s}}}", from=1-4, to=2-4]
	\arrow["{\tilde{\Psi}}", from=2-1, to=2-4]
\end{tikzcd}\]
Since both $\pi_X$ and $\pi_Y$ are risometries on any chosen local embedding by Proposition \ref{Prop: Projection is risometry}, the situation is reduced to a $\kappa$-scheme isomorphism of affine spaces, which are base changes of $\kappa$-schemes, and so the result follows from Proposition \ref{Prop: Correspondence of rtspaces}.
\end{proof}

\subsection{Tangent Cones and Obstructions to Riso-Triviality}
Having established riso-triviality spaces and dimensions of $\O_K$-tuples on valuative balls as formal invariants of singularities, it is a natural task to study how these notions behave at the more canonical $\O_K$-tuples and valuative balls at our disposal.
\begin{definition}\label{Def: rtsp of x}
    Let $k$ be a field. The \emph{riso-triviality space of a $k$-scheme $X$ at a closed point $x$} is the maximal $\kappa(x)$-linear subspace $\mathrm{rtsp}_x(X)\subset T_xX$ such that for every admissible field extension $K/\kappa$ over $k$: $$\mathrm{rtsp}_x(X)\times_k\spec(\kappa)=\mathrm{rtsp}_{B_{X_\kappa,x}}\big(X_\kappa(\O_K)\big),$$ as an equality of $\kappa$-linear subspaces of $T_xX_\kappa$. The riso-triviality dimension of $X$ at $x$ is $$\mathrm{rtd}_x(X):=\dim_{\kappa(x)}\mathrm{rtsp}_x(X).$$
\end{definition}
We will find constraints for this kind of riso-triviality spaces. For the rest of the section, we will restrict without loss of generality to the situation where $k$ is algebraically closed and we have a fixed admissible extension $K/\kappa$ over $k$ for which $\kappa=k$.

Let $x\in X(\kappa)$, set $A=\widehat{\O_{X,x}}$, $B=\kappa\llbracket x_1,\dots,x_n\rrbracket$ and consider a radical ideal $I=(f_1,\dots,f_r)\subset B$ such that $A\simeq B/I$. We fix an injective local $\kappa$-algebra homomorphism $\kappa\llbracket t\rrbracket\hookrightarrow \O_K$. This induces a further containment:
$$\widehat{X}_x(\kappa\llbracket t\rrbracket) \subset \widehat{X}_x(\O_K).$$
This also comes with an associated injective group homomorphism $\mathbb{Z}\subset \Gamma$. We use this for making sense of $t$-adic valuation in relation to the valuative structure of $\widehat{X}_x(\O_K)$. 
\begin{definition}
     We identify $\mathrm{gr}(B)\simeq \kappa[x_1,\dots,x_n]$. The geometric tangent cone of $X$ at $x$ is: $$C_x(X):=\spec\big(\kappa[x_1,\dots,x_n]/\mathrm{in} (I)\big)\hookrightarrow T_xX\simeq \spec\big(\mathrm{gr}(B)\big),$$
    where $\mathrm{in}(I)=\langle f_\mathrm{in}~| f\in I\rangle$, and for each $f$ we denote its lowest degree part by $f_\mathrm{in}$.
\end{definition}
\begin{definition}
    For each $\gamma\in \Gamma$ consider the following ideals of $\O_K$: $$I_{\geq \gamma}:=\{r\in \O_K~|~v(r)\geq \gamma\} ~~~,~~~I_{> \gamma}:=\{r\in R~|~v(r)>\gamma\}.$$
    From the quotients $\O_K^\gamma:=I_{\geq \gamma}/I_{>\gamma}$ the valuative grading of the ring can be defined: $$\mathrm{gr}_v \O_K=\bigoplus_{\gamma\in \Gamma}\O_K^\gamma.$$ 
\end{definition}
\begin{remark}\label{rmrk: grading of R and RV}
    It follows from Definition \ref{def: RV-relation} and Example \ref{ex: RV-quot and affine embeddings}, that the RV-quotient of an arc $\alpha=(\alpha_1,\dots,\alpha_n)\in T_x{X(\O_K)}$ with valuative order $\gamma$, is determined by the image of each $\alpha_i$ in $\O_K^\gamma$.
\end{remark}

\begin{definition}
     The RV-closure of $\widehat{X}_x(\kappa\llbracket t\rrbracket)\subset \widehat{X}_x(\O_K)$ is the following set:
     \begin{align*}
\overline{\mathcal{B}}:=&~\big(\mathrm{rv}|_{\widehat{X}_x(\O_K)}\big)^{-1}\circ\mathrm{rv}\big(T_xX(\kappa\llbracket t\rrbracket)\big)
    \end{align*}
    We clearly have $\widehat{X}_x(\kappa\llbracket t\rrbracket)\subset\overline{\mathcal{B}}$, and this is generally a strict inclusion.
    \end{definition} 
    For each $\alpha\in \overline{\mathcal{B}}$, take $\nu=v({\alpha})\in \Z_{>0}$. By construction, there exists an element $\beta\in T_xX(\kappa\llbracket t\rrbracket)$ such that the truncated image $\theta_\nu(\alpha)$, viewed as an element of $T_xX(\O_K^{\leq \nu})$, fits in a commutative diagram as follows:
    \[\begin{tikzcd}
	B & {\kappa\llbracket t\rrbracket} & R \\
	& {\O_K^{\leq \nu}}
	\arrow["\beta", dashed, from=1-1, to=1-2]
	\arrow["{\theta_\nu(\alpha)}"', from=1-1, to=2-2]
	\arrow[hook, from=1-2, to=1-3]
	\arrow[two heads, from=1-3, to=2-2]
\end{tikzcd}\]
    From this, we have $\theta_\nu(\alpha)=\big(\theta_\nu(\alpha)_1,\dots,\theta_\nu(\alpha)_n\big) =\big(\theta_\nu(\beta)_1,\dots,\theta_\nu(\beta)_n\big)\in T_xX(\O_K^{\leq \nu})$ and so every coordinate of this tuple can be regarded as an element of $(t^\nu)/(t^{\nu+1})\simeq \kappa$. Using this isomorphism, denote:
    $$\psi(\alpha):=\big(\theta_\nu(\alpha)_1,\dots,\theta_\nu(\alpha)_n\big) \in T_xX(\kappa).$$
    Of course, this identification depends only on the presentation of $\widehat{\O_{X,x}}$ and the inclusion $\kappa\llbracket t\rrbracket\subset R$, so we get a well defined set theoretic map: 
    \begin{align*}
    \psi\colon\overline{\mathcal{B}}\longrightarrow T_xX(\kappa)
\end{align*}

    The restricted map $\psi|_{\widehat{X}_x(\kappa\llbracket t \rrbracket)}$ can be described intrinsically in a different way: Every arc $\gamma\colon A\to \kappa\llbracket t\rrbracket$ induces a graded map $\mathrm{gr}(\gamma)\colon \mathrm{gr} ~A\to \mathrm{gr}~\kappa\llbracket t\rrbracket \simeq \kappa[T]$. The further evaluation $T=1$ defines a $\kappa$-point of the tangent cone $C_x(X)\subset T_xX$ which is precisely $\psi(\gamma)$. From this, we see that $\psi|_{\widehat{X}_x(\kappa\llbracket t \rrbracket)}$ factors through the inclusion $C_x(X)(\kappa)\hookrightarrow T_xX(\kappa)$. The following Lemma generalizes this property.
\begin{proposition}\label{Prop: RV Determines Reduced Cone}
    The set-theoretic map $\psi\colon \overline{\mathcal{B}}\to T_xX(\kappa)$ satisfies  $\mathrm{Im}(\psi)= C_x(X)(\kappa)$.
\end{proposition}
\begin{proof}
We first verify $\mathrm{Im}(\psi)\subset C_x(X)(\kappa)$. Let $\alpha\in \overline{\mathcal{B}}$ and consider its order $\nu\in \Z\subset\Gamma$. Notice that the image of $\alpha(f)$ in $\in \O_K^{\gamma_0}$, where $\gamma_0=\deg f_\mathrm{in}\cdot \nu$, can be identified naturally with $f_\mathrm{in}\big(\theta_\nu(\alpha)_1,\dots, \theta_\nu(\alpha)_n\big)\in (t^{\gamma_0})/(t^{\gamma_0+1})\subset \O_K^{\gamma_0}$. Indeed, from the valuation properties we have for any monomial $x^u=\prod_i x_i^{u_i}$:
\begin{align*}
    v(\alpha(x^u))&=u_1v\big(\alpha(x_1)\big)+\dots+u_nv\big(\alpha(x_n)\big)\\
    &\geq (u_1+\dots+u_n)\cdot v(\alpha)\\
    &=\deg(x^u)\cdot v(\alpha)
\end{align*}
Therefore, $\alpha(f)|_{\O_K^{\gamma_0}}$ depends only on the minimal graded part of $f$ evaluated at the $\nu$-graded part of each image $\alpha(x_i)$. In this way, for any $f\in I$, this latter evaluation $f_\mathrm{in}\big(\theta_\nu(\alpha)_1,\dots, \theta_\nu(\alpha)_n\big)$ has to vanish. This says, under the isomorphism $(t^\nu)/(t^{\nu+1})\simeq \kappa$, that $\psi(\alpha)$ is a $\kappa$-point of $T_xX$ vanishing at the initial ideal of $I$. Now we verify surjectivity of $\psi\colon \overline{\mathcal{B}} \to C_x(X)(\kappa)$: Every $\lambda\in \kappa^*$ defines an automorphism $\kappa\llbracket t\rrbracket \to \kappa\llbracket t \rrbracket$, sending $t\mapsto \lambda t$. Since $\kappa$ is algebraically closed, this automatically implies that $\psi(\overline{\mathcal{B}})$ is stable by scalar multiplication. In this way, it suffices to prove that the further composition:
    $$\overline{\mathcal{B}}\stackrel{\psi}{\longrightarrow} C_x(X)(\kappa)\longrightarrow \mathbb{P}C_x(X)(\kappa)$$
    is surjective, where $\mathbb{P}C_x(X)$ stands for the projectivized tangent cone of $X$ at $x$. Let $y\in \mathbb{P}C_x(X)(\kappa)$ and consider $\sigma\colon \tilde{X}\to X$ the blow-up of $X$ at $x$. We have $E=\mathrm{Exc(\sigma)}\simeq \mathbb{P}C_x(X)$, so $y$ defines a closed point of $E$ lying over $x$. Since $\tilde{X}$ is reduced, $E$ is Cartier, and because of prime avoidance, there exists an irreducible curve $C=\spec(S)\subset \tilde{X}$ passing through $y$ not contained in $E$. Since finite extensions of $\kappa$ are trivial, one of the components of the normalization of $S$ defines an arc $\beta\colon \spec(\kappa\llbracket t\rrbracket)\to \tilde{X}$ such that $\sigma(\beta)\colon \spec(\kappa\llbracket t\rrbracket)\to X$ corresponds to a non-constant arc based at $x$. It follows from the construction that the associated graded map $\mathrm{gr}(\sigma(\beta))\colon \mathrm{gr}A\to \mathrm{gr}~\kappa\llbracket t\rrbracket\simeq \kappa[T]$ defines a non-trivial $\kappa$-point of $C_x(X)$, via evaluation $T=1$, supported on the line that corresponds to $y\in \mathbb{P}C_x(X)$. Surjectivity of $\psi$ follows from the fact that $\psi(\sigma(\beta))$ is precisely the $\kappa$-point defined in this way.
\end{proof}

\begin{theorem}\label{Prop: invariant image in the cone}
    Let $x\in X(\kappa)$, and consider $W=\mathrm{rtsp}_x(X)$ as a $\kappa$-linear subspace of $T_xX$. Then, we have:
    $$ C_x(X)(\kappa)+W(\kappa)= C_x(X)(\kappa).$$
    In particular, we have a canonical containment $W\subset C_x(X)$ and a dimensional bound: $$\mathrm{rtd}_x(X)\leq \dim X.$$
\end{theorem}
\begin{proof}
Note that there is a set theoretic function $\overline{\psi}\colon \mathrm{RV}_{\O_K}(T_xX)\to T_xX(\kappa)$ making the following diagram commute:
\[\begin{tikzcd}
	{\overline{\mathcal{B}}} & {T_xX(\kappa)} \\
	{\mathrm{RV}_{\O_K}(T_xX)}
	\arrow["\psi", from=1-1, to=1-2]
	\arrow["{\mathrm{rv}}"', from=1-1, to=2-1]
	\arrow["{\overline{\psi}}"', dashed, from=2-1, to=1-2]
\end{tikzcd}\]
Indeed, the image of an arc $\alpha\in\overline{\mathcal{B}}$ depends only on its tuple decomposition and the grading of $\O_K$, thus from the same reasoning as in Remark \ref{rmrk: grading of R and RV} we get that the map $\psi$ factors through the RV-quotient $\mathrm{rv}\colon \overline{\mathcal{B}}\to \mathrm{RV}_{\O_K}(T_xX)$, and so the extension $\overline{\psi}$ exists. Consider the associated projection $\pi_W\colon T_x{X(\O_K)} \to W(\O_K)$. By Proposition \ref{Prop: Improved riso} there exists a $W$-straightener $\varphi\colon B_{X,x}\to C=(\pi_W|_{B_{X,x}})^{-1}(x)+B_{W,x}$ adapted to $W$ and preserving RV-classes. In this way, we get the following commutative diagram:
\[\begin{tikzcd}
	{B_{X,x}} & {\overline{\mathcal{B}}} & {T_xX(\kappa)} \\
	C && {\mathrm{RV}_{\O_K}(T_xX)}
	\arrow["\varphi"', from=1-1, to=2-1]
	\arrow["{{\mathrm{rv}}}", from=1-1, to=2-3]
	\arrow[hook', from=1-2, to=1-1]
	\arrow["\psi", from=1-2, to=1-3]
	\arrow["{{\mathrm{rv}}}"', from=2-1, to=2-3]
	\arrow["{{\overline{\psi}}}"', from=2-3, to=1-3]
\end{tikzcd}\]
Recall that $C$ is $W$-translation invariant on $B_{T_xX,0}$ with the further restriction $(\pi_W|_C)\circ \varphi=\pi_W|_{B_{X,x}}$. Thus, for every $\alpha\in \overline{\mathcal{B}}\subset B_{X,x}$ and $w\in W(\kappa)$, the element $\varphi(\alpha)+t^{v(\alpha)}w$ is by construction an element of $\varphi\big(\overline{\mathcal{B}}\big)$ satisfying:
\begin{align*}
    \psi\circ\varphi^{-1}(\varphi(\alpha)+t^{v(\alpha)}w)&=\overline{\psi}\circ\mathrm{rv}(\varphi(\alpha)+t^{v(\alpha)}w)\\
    &=\overline{\psi}\circ \mathrm{rv}(\alpha+t^{v(\alpha)}w)\\
    &= \psi(\alpha)+w.
\end{align*}
This concludes $\psi\big(\overline{\mathcal{B}}\big)+W(\kappa)\subset \psi\big(\overline{\mathcal{B}}\big)$, and the result follows from Proposition \ref{Prop: RV Determines Reduced Cone}. 
\end{proof}
\begin{remark}
    Notice that the construction of the rv-closure $\overline{\mathcal{B}}$ is strictly necessary for our proof of Corollary \ref{Prop: invariant image in the cone}. Indeed, the element $\varphi^{-1}(\varphi(\alpha)+t^{v(\alpha)}w)$ which is used to produce a translation of $\psi(\alpha)$ by $w$, lies by definition in $\overline{\mathcal{B}}$ but has no reason to be contained in $\widehat{(X)_x}(\kappa\llbracket t\rrbracket)$.
\end{remark}

\begin{remark}
    The statement in Theorem \ref{Prop: invariant image in the cone}  is not enough to characterize a $\kappa$-linear subspace $W\subset T_xX$ as the riso-triviality space $\mathrm{rtsp}_x(X)$. See Example \ref{ex: cusp}.
\end{remark}


\section{Riso-Stratifications}
In this section we check that our riso-triviality notions are compatible with the original ones driving the riso-stratification process. This allow us to prove embedding independence of riso-stratifications and globalize their construction with a resulting canonical and functorial procedure for $k$-schemes, where $k$ is any field of characteristic 0.

\subsection{On The Compatibility of Geometric Setups}
 Let $K/\kappa$ be an admissible extension over a domain $R$ of characteristic 0. In this subsection we summarize the original geometric construction of the riso-triviality dimension invariant for subsets of $K^n$, which is the central notion underlying the riso-stratification process. Then, we prove compatibility with our notion of riso-triviality dimension for $\O_K$-tuples of $\A_\kappa^n$.

 Fix a group scheme isomorphism $\A_\kappa^n\simeq \mathbb{G}_a^m$, which yields an associated additive group isomorphism $\A_\kappa^n(K)\simeq K^n$. We use the inclusion $\A_\kappa^n(\O_K) \hookrightarrow\A_\kappa^n(K)$ to make sense of $\O_K$-tuples of $\A_\kappa^n$ as tuples of subsets of $K^n$. Just to make a distinction between the setups, we will use the adjective ``strong'' when recalling definitions from the original model-theoretic setting. 

\begin{definition}\cite{BWH25}*{2.1.12}
The valuation of $K$ is extended to $K^n$ via:
$$v\big((x_1,\dots,x_n)\big):= \min_{1\leq i\leq n}v(x_i).$$
The closed valuative ball of $K^n$ with radius $\gamma\in\Gamma$ centered at $x\in K^n$ is the set:
    $$B_{\geq \gamma}(x)=\{x'\in K^n~\big|~v(x'-x)\geq \gamma  \}.$$ Similarly, the open ball with the same radius and center is:
    $$B_{>\gamma}(x)=\{x'\in K^n~\big|~v(x'-x)> \gamma  \}.$$
\end{definition}
\begin{remark}
    We use different sub-indices which are compatible with our additive notation for the valuation, in contrast to our reference.
\end{remark}
\begin{remark}
    It is straightforward to check that our notions of valuative balls and distances are compatible with this extended valuation on $K^n$. In particular, for any $\alpha\in \A_\kappa^n(\O_K)$ and $\gamma\in\Gamma$ such that $\gamma> 0$, we have a natural comparison of closed valuative balls:
    $$B_{\A_\kappa^n,\alpha}(\gamma)=B_{\geq \gamma}(\alpha) \subset K^n.$$
    However, be careful that our slightly different definitions give a different comparison for open balls of radius $0$ centered at $\kappa$-points $x$:
    $$ B_{\A_\kappa^n,x}=B_{>0}(x).$$
\end{remark}

\begin{definition}\cite{BWH25}*{2.3.1}\label{def: old risometry}
    The \emph{strong rv-relation on $K^n$} is defined via:
    $$x\sim _\mathrm{(rv)} y \iff v(x-y)>v(x) \text{  or  }x=y.$$
    A strong risometry is a set-theoretic map $\varphi\colon A\to A'$ such that for every $x, y$ in $A$:
    $$\varphi(x)-\varphi(y)\sim_{(rv)} x-y .$$
\end{definition}
\begin{lemma}\label{lemma: risometrias coinciden}
    Let $A,A'\subset \A_\kappa^n(\O_K)$ be two infinitesimal sets and $\varphi\colon A\to A'$ a set-theoretic function. Then, $\varphi$ is a risometry in $\A_\kappa^n$ if and only if, $\varphi$ is a strong risometry in $K^n$.
\end{lemma}
\begin{proof}
    This just follows from the fact that our rv-relation $\sim_\mathrm{rv}$ from Definition \ref{def: RV-relation} coincides in $B_{\A_\kappa^n,0}=B_{> 0}(0_{K^n})$ with the restricted strong rv-relation $\sim_{(rv)}$. Indeed, from Definition \ref{Def: Val-Distance} we see that $v(x,y)=v(x-y)$ for each $x,y\in B_{\A_\kappa^n,0}$, where the latter can be considered as the restricted valuation from $K^n$. 
\end{proof}

\begin{definition}\cite{BWH25}*{2.3.6}
    Let $B\subset K^n$ be a valuative ball, $\mathcal{A}$ a non-empty finite tuple of subsets of $B$ and $\tilde{W}\subset K^n$ a $K$-vector subspace. We say that $\mathcal{A}$ is \emph{$\tilde{W}$-translation-invariant on $B$} if for every $x,x'\in B$ with $x-x'\in \tilde{W}$ we have:
    $$x\in A\iff x'\in A,$$
    for each $A\in \mathcal{A}$.
\end{definition}
\begin{definition}\cite{BWH25}*{2.3.8 and 2.3.13}\label{def: old riso-triviality}
    Let $B\subset K^n$ be a closed valuative ball of radius $\gamma$, $\mathcal{A}$ a finite tuple of sets and $W\subset \kappa^n$ a $\kappa$-linear subspace. We say that $\mathcal{A}$ is strongly $W$-riso-trivial on $B$ if there exists a strong risometry $\tilde{\varphi}\colon B\to B$ such that the tuple $\tilde\varphi(\mathcal{A}|_B)$ is $W(K)$-translation-invariant on $B$. The riso-triviality dimension of $\mathcal{A}$ on $B$ is: $$\underline{\mathrm{rtd}}_B(\mathcal{A}):=\max_{\underset{\mathcal{A}\textit{ is strongly $V$-riso-trivial on $B$}}{V\subset \kappa^n}}\dim_\kappa V.$$
\end{definition}
\begin{remark}
    We are using \cite{BWH25}*{2.3.9} to give a slightly modified but equivalent version of the definition of $W$-riso-triviality. That is, we are fixing a particular lift $\tilde{W}=W(\O_K)$ of the $\kappa$-linear subspace $W(\kappa)\subset \kappa^n$ with respect to the residue map defined in \cite{BWH25}*{2.3.7}. We also omit the necessity of existence of riso-triviality spaces for this setup.
\end{remark}
\begin{remark}
    Because we are specializing to the algebraic situation, we only define translation-invariantness and riso-triviality for sets, and omit the use of indicator functions \cite{BWH25}*{2.3.5}.
\end{remark}
Our definition of $W$-riso-triviality for $\mathcal{A}$ on a valuative ball $B\subset \A_\kappa^n(\O_K)$ is less restrictive, in the sense that it allows $\mathrm{supp}(\mathcal{A}|_B)$ as a domain of the corresponding $W$-straightener instead of the whole $B$. This change of perspective, which was indispensable for Proposition \ref{Prop: Projection is risometry}, recovers the same information:

\begin{lemma}\label{Lemma: Risometrias-extendibles} Let $A,A'\subset B_{\A_\kappa^n,x}(\gamma)$ be two subsets and $\varphi\colon A\to A'$ a given risometry. Then, there exists a risometry $\tilde{\varphi}\colon B_{\A_\kappa^n,x}(\gamma)\to B_{\A_\kappa^n,x}(\gamma)$ extending $\varphi$.
\end{lemma}
\begin{proof}
    Set 
 $\mathcal{F}_\varphi:=\left\{
     f \ \middle| \ f \ \text{\small is a risometry extending }\varphi\ \text{\small and } 
    \operatorname{dom}(f),\ \operatorname{codom}(f) \subset B_{\A_\kappa^n,x} 
    \right\},$
    and consider the partial order
    $f \prec g \iff \operatorname{dom}(f) \subset \operatorname{dom}(g) 
    \ \text{and} \ g|_{\operatorname{dom}(f)} \equiv f$.
    Every ascendant chain $\{f_i\}_{i\in I}$ in $\mathcal{F}_\varphi$ is bounded above by the risometry 
    $\psi\colon \bigcup_{i\in I} \operatorname{dom}(f_i) \to \bigcup_{i\in I} \operatorname{codom}(f_i),$
    defined by $x \mapsto f_j(x)$ for any $j \in I$ with $x \in \operatorname{dom}(f_j)$. Thus, by Zorn's Lemma, we get a maximal object $\tilde{\varphi} \in \mathcal{F}_\varphi$. Since we can use spherical completeness to get limits of Cauchy sequences, the equality $\mathrm{dom}(\tilde{\varphi}) = B_{\A_\kappa^n,x}$ is forced: Every $g \in \mathcal{F}$ can be extended to a risometry with valuatively closed domain by valuative continuity. Once we have closed domains, every $x_0 \in B_{\A_\kappa^n,x} \setminus \mathrm{dom}(g)$ has an associated $y \in \mathrm{dom}(g)$ such that 
    $v(x_0,y) = \max_{x \in \operatorname{dom}(\tilde{\varphi})} v(x_0,x) < \infty$. In this situation, it suffices to set $x_0 \mapsto x_0 + \tilde{\varphi}(y)-y$ to get an admissible extension of $g$. Finally, we must have $\operatorname{codom}(\tilde{\varphi}) = B_{\A_\kappa^n,x}$ since the inverse risometry cannot be extended. 
\end{proof}

\begin{corollary}\label{Cor: rtd son iguales}
    Let $\mathcal{A}$ be a $\O_K$-tuple of $\A_\kappa^n(\O_K)$ and $B$ an infinitesimal valuative ball of $\A_\kappa^n$ such that $\mathcal{A}|_B$ is non-trivial. Then, $\mathcal{A}$ is $W$-riso-trivial on $B$ if and only if, $\mathcal{A}$ is strongly $W$-riso-trivial on $B$. In particular:
    $$\underline{\mathrm{rtd}}_B(\mathcal{A})=\mathrm{rtd}_{B}(\mathcal{A}).$$
    That is, both notions of riso-triviality dimension coincide when applied to infinitesimal balls centered at $\O_K$-points of the support of $\mathcal{A}$.
\end{corollary}
\begin{proof}
    Fix $x\in\mathrm{supp}(\mathcal{A})$ and $B=B_{\geq \gamma}(x)$ with $\gamma\ge 0$.
    Suppose $\mathcal{A}$ is strongly $W$-riso-trivial on $B$ and let $\tilde{\varphi}\colon B\to B$ be a risometry such that the tuple $\tilde\varphi(\mathcal{A})$ is $W(K)$-riso-trivial on $B$. Any pair $x,y\in B$ satisfying $x-y\in W(K)$ implies the refined $x-y\in B_{W,0}(\gamma)$. From this, it follows that $\tilde\varphi(A)+B_{W,0}(\gamma)=\tilde{\varphi}(A)$ for every $A\in \mathcal{A}$ and so $\tilde{\varphi}(\mathcal{A})$ is $W$-translation-invariant on $B$ and the restricted $\tilde{\varphi}|_{\mathrm{supp}(\mathcal{A})}$ is a $W$-straightener for $\mathcal{A}$ on $B$. For the other implication, suppose $\mathcal{A}$ is $W$-riso-trivial and consider $\varphi\colon \mathrm{supp}(\mathcal{A})\to C$ a $W$-straightener for $\mathcal{A}$ on $B$. Consider now an extension $\tilde\varphi\colon B\to B$ of $\varphi$, which exists by Lemma \ref{Lemma: Risometrias-extendibles}. Notice that for every $x,y\in B$ such that $x-y\in W(K)$, we have the improved $x-y\in B_{W,0}(\gamma)$. Since this ball is defining locally $W$-translation-invariantness of $\varphi(\mathcal{A})$, by construction we have $x\in \tilde{\varphi}(A)=\varphi(A)$ if and only if $y\in \tilde{\varphi}(A)$, for every $A\in\mathcal{A}$. This verifies that both riso-triviality notions are the same for $\mathcal{A}$, and the result about riso-triviality dimensions follows directly from this.
\end{proof}

\begin{corollary}\label{Cor: rtdsp independent of K}
    Let $k$ be a field, $X$ a $k$-scheme and $x\in X(\kappa)$. Then, for every admissible field extension $K/\kappa$ over $k$ we have: $$\mathrm{rtsp}_x(X)\times_k\spec(\kappa)=\mathrm{rtsp}_{B_{X_\kappa,x}}\big(X_\kappa(K)\big)\subset T_xX_\kappa.$$ That is, the riso-triviality space of $X$ at $x$ can be computed from a unique admissible field extension $K/\kappa$ over $k$.
\end{corollary}
\begin{proof}
    Once the definitions of riso-triviality agree in embedded situations, we can in particular assume the consequences of \cite{BWH25}*{3.1.3} about definability of riso-triviality spaces.
\end{proof}

\subsection{Embedding Independence of Riso-Stratifications}
We recall the technical definition of riso-stratifications involving the shadow iteration construction. Then, we prove embedding independence of this riso-stratification process when applied to algebraic objects. 

Let $K/\kappa$ be an admissible extension over a domain $R$ of characteristic 0. Consider $\mathscr{L}_0$ the language of rings expanded by constants for elements of $R$, $\mathscr{L}$ an expansion of $\mathscr{L}_0$ by a predicate for the valuation ring $\O_K$ and the theory $\mathscr{T}=\mathrm{Th}_\mathscr{L}(K)$. In this setting, we satisfy the assumptions of \cite{BWH25}*{4.1.1} and so we can safely apply their results. 

Let $X$ be an $R$-scheme and $X\hookrightarrow \A_R^n$ a closed embedding. Denote $X_\kappa=X\times_R \kappa$ as a $\kappa$-scheme which is closed in $\A_\kappa^n$. Every fixed coordinate choice for $\A_R^n$ lets us regard $X_\kappa(K)$ as an $\mathscr{L}_0$-definable subset of $K^n$ for which we can run the riso-stratification process. Let us recall this construction in detail:

\begin{definition}\cite{BWH25}*{4.2.3} \label{def: shadow iteration}
    Let $\mathcal{A}$ be a finite tuple of $\mathscr{L}$-definable sets of $K^n$. The \emph{shadow of $\mathcal{A}$} is the partition of $\kappa^n$ into $\mathscr{L}_0$-definable sets $\big(\mathrm{Sh}_{d}(\mathcal{A})\big)_{0\leq d\leq n}$ determined by:
    \begin{align*}
        z\in \mathrm{Sh}_{d}(\mathcal{A})(\kappa)&\iff  \underline{\mathrm{rtd}}_{B_{>0}(z)}(\mathcal{A})=d
    \end{align*}
    \cite{BWH25}*{4.3.1} Set $\mathrm{Sh}_d^0(\mathcal{A}):=\mathcal{A}$ for $0\leq d\leq n$ and define recursively for every $s\geq 1$ the $s$-th shadow of $\mathcal{A}$ via:
    $$\big(\mathrm{Sh}_{d}^s(\mathcal{A})\big)_{0\leq d\leq n}=\bigg(\mathrm{Sh}\Big(\mathcal{A}\cup \big\{\{\big(\mathrm{Sh}_i^{t}(\mathcal{A})\big)(K)\}~|~0\leq i\leq n~,~t<s   \big\}\Big)\bigg)_{0\leq d\leq n}.$$
\end{definition}
\begin{remark}
    The fact that shadow construction produces $\mathscr{L}_0$-definable strata is non-trivial and follows from \cite{BWH25}*{3.1.3}. This property is required for making sense of iterated shadows.
\end{remark}
\begin{remark}
    To each tuple $\mathcal{A}$ of $\mathscr{L}$-definable sets, its shadow produces a stratification of the whole ambient space $\A_\kappa^n$ by $\mathscr{L}_0$-definable pieces. In the case $\mathcal{A}$ is a tuple composed of $\mathscr{L}_0$-definable sets, then the resulting shadow restricts to a geometric stratification of $\mathrm{supp}(\mathcal{A})$ as a $\mathscr{L}_0$-definable set. In order to retrieve this, it is only necessary to forget the strata of dimension $>\dim\big( \mathrm{supp}(\mathcal{A})\big)$.    
\end{remark}

For the given embedding $X\hookrightarrow \A_R^n$ and fixed choice of admissible extension $K/\kappa$ over $R$, the first shadow of the singleton $\mathcal{A}=\{X_\kappa(K)\}$ is already a good approximation of a stratification of $X$. However, this might not be composed of Zariski closed partial unions \cite{BWH25}*{4.2.4}. The process whose iteration can be used to get this desired property is the following:
\begin{definition}(cf. \cite{BWH25}*{4.4.1})\label{def: riso-strat}
    The \emph{riso-stratification of $X$ with respect to $\psi\colon X_\kappa\hookrightarrow\mathbb{A}^n_R$} is the partition of $X$ into reduced constructible strata: $$\big(\mathcal{S}^\psi_d(X_\kappa)\big)_{0\leq d\leq \dim (X/R)}:=\Big(\psi^{-1}\mathrm{Sh}_{d}^n\big(\{\psi_*X(K)\}\big)\Big)_{0\leq d\leq \dim X_\kappa}.$$
    Where $\dim (X/R)=\dim X_\kappa$ denotes the relative dimension of $X\to \spec (R)$. This makes sense since the shadow iteration process returns $\mathscr{L}_0$-definable strata. 
\end{definition}
\begin{remark}\label{rmrk: locally closed shadow is riso-strat}
    It follows from a more detailed description of the shadow construction \cite{BWH25}*{4.4.5}, that the riso-stratification of $X$ with respect to $\psi$ equals its first shadow, if and only if every partial union of strata of the first shadow is Zariski closed on $X$. Moreover, for every $z\in \mathcal{S}^\psi_d(X)(\kappa)$ we always have:
    $$\underline{\mathrm{rtd}}_{B_{>0}(z)}\big(X_\kappa(K)\big) \geq d .$$
\end{remark}
The following Theorem summarizes the main properties which are already known to be satisfied by these embedded riso-stratifications:
\begin{theorem}\label{thm: propiedades de riso-strat}
       The riso-stratification of $X$ with respect to $\psi\colon X\hookrightarrow \A_\kappa^n$ satisfies the properties:
       \begin{enumerate}
           \item Each finite union of strata $\mathcal{S}^\psi_{\leq d}(X)$ is Zariski closed in $X$ \cite{BWH25}*{4.4.11}.\vspace{0.5ex}
           \item each stratum $\mathcal{S}^\psi_d(X)$ is pure of relative dimension $d$ and smooth after base change to $\kappa$ \cite{BWH25}*{4.6.8}.\vspace{0.5ex}
           \item The stratification satisfies Whitney conditions after base change to $\kappa$, in the sense of \cite{BWH25}*{4.6.4}.\vspace{0.5ex}
           \item Local motivic Poincar\'e series are trivial along each stratum after base change to $\kappa$, in the sense of \cite{BWH25}*{5.4.4(2)}\vspace{0.5ex}
           \item It depends only on the $\mathscr{L}_0$-formula that defines $X_\kappa(K)\subset K^n$. That is, the stratification does not depend on the choice of admissible extension $K/\kappa$ over $R$ \cite{BWH25}*{4.1.7}. 
       \end{enumerate}
\end{theorem}
Now, we combine this construction with our methods to detect embedding independence in a strong sense:

\begin{theorem}
\label{Thm: Local Riso-Strat Commutes}
Let $Y$, $Z$ be affine $R$-schemes and $f\colon Y\to Z$ be a smooth morphism of relative dimension $N$. Fix two closed embeddings $\psi_1\colon Y\hookrightarrow \A_R^{n_1} $ and $\psi_2\colon Z\hookrightarrow \A_R^{n_2}$, and consider the associated riso-stratifications $\big(\mathcal{S}^{\psi_1}_d(Y)\big)_{0\leq d\leq \dim (Y/R)}$ and $\big(\mathcal{S}^{\psi_2}_{\rho}(Z)\big)_{0\leq \rho\leq \dim (Z/R)}$. Then, we have $\mathcal{S}_r^{\psi_1}(Y)=\varnothing$ for $r<N$ and for every $0\leq \rho \leq \dim (Z/R)$ we have:
$$\mathcal{S}^{\psi_1}_{N+\rho}(Y)= f^{-1} \mathcal{S}^{\psi_2}_{\rho}(Z),$$
as locally closed subsets of $Y$. In particular, the embedded riso-stratification of $Z$ determines the embedded riso-stratification of $Y$.
\end{theorem}
\begin{proof}
    For the fixed admissible extension $K/\kappa$ over $R$, we consider the base changes $Y_\kappa$ and $Z_\kappa$ and regard $Y_\kappa(K)\subset K^{n_1}$ and $Z_\kappa(K)\subset K^{n_2}$ via the respective embeddings. We prove the Proposition via verifying the following relation:
    \begin{equation}\label{eq: shadow induction}
        \mathrm{Sh}^s_{N+\rho}\big((\psi_1)_*(Y_\kappa)(K)\big)= f^{-1} \mathrm{Sh}^s_{\rho}\big((\psi_2)_*(Z_\kappa)(K)\big),
    \end{equation}
    for every $s\geq 0$ and $0\leq \rho\leq \dim Z$, by strong induction on $s$. From there, the stabilizing property \cite{BWH25}*{4.3.1} would conclude the result. For the base case, we have only one type of stratum and trivially for the non-stratified $\mathscr{L}_0$-definable sets $Y=f^{-1}Z$. Now, assume that (\ref{eq: shadow induction}) holds up to a fixed $s\geq 0$ and consider the tuples: 
    \begin{align*}
        \mathcal{A}_1&=\big((\psi_1)_*Y_\kappa(K),\mathrm{Sh}^1_0((\{(\psi_1)_*Y_\kappa(K)\})(K),\dots ,\mathrm{Sh}^s_{n_1}(\{(\psi_1)_*Y_\kappa(K)\})(K)\big)\\
        \mathcal{A}_2&=\big((\psi_2)_*Z_\kappa(K),\mathrm{Sh}^1_0((\{(\psi_2)_*Z_\kappa(K)\})(K),\dots ,\mathrm{Sh}^s_{n_2}(\{(\psi_2)_*Z_\kappa(K)\})(K)\big)
    \end{align*} of $K^{n_1}$ and $K^{n_2}$ respectively. We take them such that next shadows in question can be detected via $\kappa$-points as follows:
    \begin{align*}
    y\in \mathrm{Sh}^{s+1}_{d}\big((\psi_1)_*Y_\kappa\big)(\kappa)&\iff  \underline{\mathrm{rtd}}_{B_{>0}(y)}(\mathcal{A}_1)=d\\
    z\in \mathrm{Sh}^{s+1}_{\rho}\big((\psi_2)_*Z_\kappa\big)(\kappa)&\iff  \underline{\mathrm{rtd}}_{B_{>0}(z)}(\mathcal{A}_2)=\rho
    \end{align*}
    In this way, it suffices to check that for every $z=f(y)\in Z_\kappa(\kappa)$ with $y\in Y_\kappa(\kappa)$ there is an equality of riso-triviality dimensions as follows:
    \begin{equation}\label{eq: ecuacion smooth morph}
        N + \underline{\mathrm{rtd}}_{B_{>0}(y)}(\mathcal{A}_1)= \underline{\mathrm{rtd}}_{B_{>0}(z)}(\mathcal{A}_2).
    \end{equation}
    This can be computed from our notion of riso-triviality dimensions applied to non-embedded $\O_K$-tuples. Indeed, it follows from Corollary \ref{Cor: rtd son iguales} that we have:
    \begin{align*}
        \underline{\mathrm{rtd}}_{B_{>0}(y)}(\mathcal{A}_1)&=\mathrm{rtd}_{B_{\A_\kappa^{n_1},y}}({\mathcal{A}}_1)=\mathrm{rtd}_{B_{Y_\kappa,y}}({\mathcal{A}}_1)\\
        \underline{\mathrm{rtd}}_{B_{>0}(z)}(\mathcal{A}_2)&=\mathrm{rtd}_{B_{\A_\kappa^{n_2},z}}(\mathcal{A}_2)=\mathrm{rtd}_{B_{Z_\kappa,z}}({\mathcal{A}}_2)
    \end{align*}
    where the equalities on the right, reflecting the intrinsicness of these invariants, come from Proposition \ref{Prop: Correspondence of rtspaces}. In this way, we get reduced via Corollary \ref{Cor: rtd invariante formal} to a verification of equation (\ref{eq: ecuacion smooth morph}) at the formal neighborhood level. Assume that $z\in Z(\kappa)\subset \A_\kappa^{n_2}(\kappa)$ is the origin. Recall that $\kappa$ being algebraically closed and $f$ being a smooth induces after base change an isomorphism: $$\widehat{\O_{Z,z}}\llbracket l_1,\dots,l_N\rrbracket \simeq \widehat{\O_{Y,y}}.$$
    Without loss of generality $n_2\gg N$ and there are linear coordinates $l_1,\dots,l_{n_1}$ such that $f^\#$, the local ring homomorphism induced by $f$ after base change along $\kappa$, fits in a commutative diagram as follows:
    \[\begin{tikzcd}
	{\kappa\llbracket l_1,\dots,l_{n_2}\rrbracket} && {\widehat{\O_{Z,z}}} \\
	& {\widehat{\O_{Z,z}}\llbracket l_1,\dots,l_N\rrbracket}
	\arrow[two heads, from=1-1, to=1-3]
	\arrow[two heads, from=1-1, to=2-2]
	\arrow["{f^\#}"', from=2-2, to=1-3]
\end{tikzcd}\]
    In this situation, the induction hypothesis (\ref{eq: shadow induction}) holding for each $s'\leq s$ and the above formal isomorphism can be used to regard $\mathcal{A}_1$ as a $\O_K$-tuple of $\A_k^{n_2}$ for which:
    $$ {\mathcal{A}}_1=  \{A + W(\O_K)~|~ A\in{\mathcal{A}}_2\},$$
    where $W=\spec(\kappa[l_1,\dots,l_N])$ as a $\kappa$-linear subspace of $\A_\kappa^{n_2}$ and $\gamma$ is the common radius of ${\mathcal{A}}_1$ and ${\mathcal{A}}_2$. Therefore, we are allowed to use Corollary \ref{Cor: riso-triv en familias triviales} which verifies equation (\ref{eq: ecuacion smooth morph}) and finishes the induction.
\end{proof}
\begin{corollary}\label{cor: riso-strat no depende delembedding}
    Let $X$ be an affine $R$-scheme. There is a unique riso-stratification of $X$, which is stable by scalar extensions. That is, for every pair of closed embeddings $\psi_1\colon X\hookrightarrow \A_k^{n_1}$ and $\psi_2\colon X\hookrightarrow \A_k^{n_2}$ we have:
    $$\big(\mathcal{S}^{\psi_1}_d(X)\big)_{0\leq d\leq \dim (X/R)}=\big(\mathcal{S}^{\psi_2}_d(X)\big)_{0\leq d\leq \dim (X/R)},$$
    component-wise as locally closed subsets of $X$. Furthermore, for every ring extension $R\subset S$ we have:
    $$\Big(\mathcal{S}^{\psi_1}_d\big(X_S\big)\Big)_{0\leq d\leq \dim (X_S/S)} = \Big(\mathcal{S}^{\psi_1}_r\big(X\big)\times_R\spec(S)\Big)_{0\leq d\leq \dim (X_S/S)} ,$$
    component-wise, as locally closed subsets of $X_S=X\times_R\spec(S)$.
\end{corollary}
\begin{proof}
    The independence of embedding follows from applying Theorem \ref{Thm: Local Riso-Strat Commutes} to the identity morphism $\mathrm{id}_{X}$. The scalar extension compatibility is already known for a fixed embedding $\psi\colon X\hookrightarrow \A_R^n$: We may assume that the field $\mathrm{Frac}(S)$ sits in a tower $K/\kappa/\mathrm{Frac}(S)/\mathrm{Frac}(R)$ from the very start, according to part (5) of Theorem \ref{thm: propiedades de riso-strat}. Then, we have trivially for the input of the riso-stratification:
    $(X_S)_\kappa(K)=X_\kappa(K)\subset K^n,$
    and so the $\kappa$-points of the resulting pieces of the riso-stratification of $X_S$ are $\mathscr{L}_0$-definable.
\end{proof}

\begin{remark}
    We are proving a slightly stronger fact: Each iterated shadow on the construction of the riso-stratification satisfies the properties of Corollary \ref{cor: riso-strat no depende delembedding}.
\end{remark}

\subsection{Global and Functorial Riso-Stratifications}
We are now in position to define a canonical and global riso-stratification process. We perform this over a field $k$ of characteristic 0. The strength of the model theoretic results of \cite{BWH25}, together with our previous results, will be converted to functorial properties satisfied by this global procedure.  
\begin{definition}
    Let $X$ be a $k$-scheme. For every $0\leq d\leq \dim X$, the \emph{$d$-th riso-stratifying sheaf of $X$} is the sheaf of ideals defined by:
$$\mathscr{I}_{X,d}(U):=I_{\mathcal{S}^{\psi}_{\leq d}(U)}\subset \O_X(U),$$
for every affine open $U\subset X$, fixed closed embedding $\psi\colon U\hookrightarrow\A_\kappa^n$ and admissible extension $K/\kappa$ over $k$. This association produces radical ideals which are well defined because of Corollary \ref{cor: riso-strat no depende delembedding} and properties (2) and (5) in Theorem \ref{thm: propiedades de riso-strat}. The definition of $\mathscr{I}_{X,d}$ at affine opens indeed determines a unique sheaf on $X$ since compatibility with restrictions follows from Theorem \ref{Thm: Local Riso-Strat Commutes}.
\end{definition}
To every $k$-scheme $X$ we are associating a canonical nested sequence of sheaf of ideals:
\begin{equation*}
    \mathscr{I}_{X,\dim X}\subseteq\dots\subseteq \mathscr{I}_{X,0}
\end{equation*}
This is exactly what we use to globalize riso-stratifications.
\begin{definition}
    The \emph{riso-stratification of $X$ over $k$} is the canonical tuple: $$\mathcal{S}_\bullet(X/k)=\{\mathcal{S}_d(X/k)\}_{0\leq d\leq \dim X},$$ of locally closed sub-schemes of $X$ defined by  $\mathcal{S}_0(X/k)=V\big(\mathscr{I}_{X,0}\big)$ and recursively for $0<d\leq \dim X$:
    $$\mathcal{S}_d(X/k)= V(\mathscr{I}_{X,d})\setminus V(\mathscr{I}_{X,{d-1}}).$$
\end{definition}
\begin{remark}
Since this is just a globalization of the procedure introduced in \cite{BWH25}, every local property from Theorem \ref{thm: propiedades de riso-strat} is inherited for free to the general construction. In particular, each stratum $\mathcal{S}_d(X/k)$ is either empty or smooth of pure dimension $d$, $\mathcal{S}_{\dim X}(X/k)=X^\mathrm{sm}$.
\end{remark}
\begin{remark}
    Riso-stratifications are compatible with restrictions to subschemes of the same dimension. That is, if $Y\subset X$ is a subscheme, of finite type over $k$ and of pure dimension $\dim Y=\dim X$, then $\mathcal{S}_d(X/k)|_Y=\mathcal{S}_d(Y/k)$ for every $0\leq d \leq \dim Y$. This follows easily from Theorem \ref{Thm: Local Riso-Strat Commutes}. This can be used to extend the definition to schemes which are essentially of finite type over $k$. 
\end{remark}
We finalize our work by summarizing the functorial properties that this globalized stratification process satisfies:
\begin{theorem}\label{thm: riso-strat are funcotiral}
    Riso-stratifications are compatible with base change along smooth morphisms and field extensions. That is, for every smooth morphism of $f\colon Y\to X$ over $k$ of relative dimension $N$, we have $\mathcal{S}_{r}(Y/k)=\varnothing$ for $r<N$ and: 
        $$ \mathcal{S}_{N+d}(Y/k)=f^{-1}\mathcal{S}_d(X/k),$$
        for every $0\leq d\leq \dim X$. In particular, riso-stratifications commute with \'etale localizations.
    Similarly, for every field extension $F/k$ we have:
        $$\mathcal{S}_d(X_F/F)=\mathcal{S}_d(X/k)\times_k\spec(F),$$
        for every $0\leq d\leq \dim X$, as $F$-schemes.
\end{theorem}
\begin{proof}
    Both properties are local on $X$, so we may take $X$ and $Y$ to be affine. The first assertion follows easily from Theorem \ref{Thm: Local Riso-Strat Commutes} under any local embedding, and the second assertion is a direct consequence of the scalar extension compatibility on Corollary \ref{cor: riso-strat no depende delembedding}.
\end{proof}

\begin{theorem}\label{Prop: Speading properties}
    Riso-stratifications have good spreading properties. More precisely, if $X$ is a $k$-scheme and $R$ is a domain such that $k=\mathrm{Frac}(R)$ and $\tilde{X}$ is an $R$-scheme such that $X\simeq \tilde{X}\times_R \spec(k)$, then each stratum $\mathcal{S}_d(X/k)\subset X$ is the base change of a uniquely determined locally closed subscheme ${\mathcal{S}}_d(\tilde{X})$ that is of finite presentation over $R$.
\end{theorem}
\begin{proof}
    This is a local property, so we can assume that $X$ is affine. In this case, the result is just a consequence of the scalar extension compatibility property of Corollary \ref{cor: riso-strat no depende delembedding} applied to the local riso-stratification of $\tilde{X}$ as an $R$-scheme.
\end{proof}

\section{Examples and Further Questions}

It is in general a highly non-trivial task to find the riso-stratification of a given $k$-scheme. As pointed out in \cite{BWH25}*{1.2}, there is not a clear path for avoiding the model-theoretic arguments that guarantee definability of the shadow constructions. That is to say, riso-stratifications \ref{def: riso-strat} being well defined is a highly non-trivial fact, which so far is only verified by the model theoretic side. Nonetheless, it is possible to compute this for some particularly simple geometric objects, following the general techniques used to deal with riso-triviality conditions.

For the following computations, we fix $k=\kappa$ and algebraically closed field of arbitrary characteristic and consider an admissible extension $K/\kappa$ over $k$.

\begin{example}\label{ex: cusp}
    Consider the cusp $C=\spec \big(k[x,y]/(y^2-x^3)\big)$. We claim that:
    \begin{align*}
    \mathcal{S}_1(C/k)&=\mathrm{Sh}^1_1\big(C(K)\big)=C_\mathrm{sm}\\
        \mathcal{S}_0(C/k)&=\mathrm{Sh}^1_0\big(C(K)\big)=C_\mathrm{sing}
    \end{align*}
    In other words, the first shadow of $C$ already stabilizes and computes the riso-stratification of $C$. 
\end{example}
\begin{proof}
    It follows from Corollary \ref{cor: smooth implies riso-trivial} and Remark \ref{rmrk: locally closed shadow is riso-strat} that $C_\mathrm{sm}\subset \mathcal{S}_1(C/k)$. On the other hand, for the singular point $x\in C$ we get from Theorem \ref{Prop: invariant image in the cone} that $\mathrm{rtsp}_x(C,\O_K)$ is either $W=\{y=0\}$ or trivial. The first case can be dismissed via Proposition \ref{Prop: Improved riso}, since the central fiber of $(\pi_W|_{B_{C,x}})\colon B_{C,x}\to W(\O_K)$ is the only one composed of a single arc. This confirms the first shadow computation, which coincides with the riso-stratification since it defines locally closed strata.
\end{proof}

\begin{example}\cite{Lipman:2000}*{1.c}
    Consider the surface $X=\spec\big(k[x,y,z]/(f)\big)$, where $f=x(x-y)(x+y)(xz-y)$. We claim that:
    \begin{align*}
        \mathcal{S}_2(X/k)&=\mathrm{Sh}^2_1\big(X(K)\big)=X_\mathrm{sm}\\\mathcal{S}_1(X/k)&=\mathrm{Sh}^1_1\big(X(K)\big)=X_\mathrm{sing}\setminus\{p_1,p_2,p_3\}\\
        \mathcal{S}_0(X/k)&=\mathrm{Sh}^1_0\big(X(K)\big)=\{p_1,p_2,p_3\}
    \end{align*}
    where $p_1=(0,0,0),~p_2=(0,0,-1),~p_3=(0,0,1)$ as closed points of $X$.
\end{example}
\begin{proof}
It follows from Corollary \ref{cor: smooth implies riso-trivial} and Remark \ref{rmrk: locally closed shadow is riso-strat} that $X_\mathrm{sm}\subset \mathcal{S}_2(X/k)$. Notice that $V=\{x=y=0\}=X_\mathrm{sing}\subset \A_k^3$, and $X(K)$ is at most $V$-riso-trivial on $B_{X,x_0}$ for a singular point $x_0\in X_\mathrm{sing}$ according to \ref{Prop: invariant image in the cone}. Consider $x_0=(0,0,z_0)\in X_\mathrm{sing}(k)\setminus \{p_1,p_2,p_3\}$, we have an embedded set-theoretic description:
\begin{align}\label{eq: decomp-example}
    B_{X,x_0}=B_{H_1,x_0}\cup B_{H_2,x_0}\cup B_{H_3,x_0}\cup B_{S,x_0},
\end{align}
where $H_1=\{x=0\},H_2=\{x-y=0\}, H_3=\{x+y=0\}, S=\{xz=y\}\subset \A_k^3$. Using this decomposition, it is possible to define the function:
\begin{align*}
    \varphi\colon B_{X,x_0}&\longrightarrow \A_k^3(\O_K)\\
    \alpha=(\alpha_1,\alpha_2,\alpha_3)&\longmapsto \begin{cases}
        (\alpha_1,\alpha_1z_0,\alpha_3)~,~\text{if $\alpha\in S(\O_K)$}\\
        \alpha~,~\text{else.}
    \end{cases}
\end{align*}
It is an exercise to prove that this map defines a risometry to its image, which is $V$-translation-invariant on $B_{\A_k^3,x}$. This confirms $X_\mathrm{sing}\setminus\{p_1,p_2,p_3\}\subset \mathcal{S}_1(X/k)$. The same decomposition (\ref{eq: decomp-example}) can be used to detect impossibility of a risometry as in Proposition \ref{Prop: Improved riso} adapted to the projection $\pi_V\colon B_{X,p_i}\to V(\O_K)$. Indeed, there are more rv-classes of differences away from the central fiber. This confirms that $\{p_1,p_2,p_3\}\subset \mathcal{S}_0(X/k)$, and concludes that the first shadow has locally closed strata, hence it already computes the riso-stratification.
\end{proof}

\begin{example}[Whitney Umbrella] Consider the surface $X=\spec\big(k[x,y,z]/(g)\big)$, where $g=x^2-y^2z$. We claim that:
    \begin{align*}
        \mathcal{S}_2(X/k)&=\mathrm{Sh}^2_1\big(X(K)\big)=X_\mathrm{sm}\\\mathcal{S}_1(X/k)&=\mathrm{Sh}^1_1\big(X(K)\big)=X_\mathrm{sing}\setminus\{p\}\\
        \mathcal{S}_0(X/k)&=\mathrm{Sh}^1_0\big(X(K)\big)=\{p\}
    \end{align*}
    where $p=(0,0,0)$ a closed point of $X$.
\end{example}
\begin{proof}
    It follows from Corollary \ref{cor: smooth implies riso-trivial} and Remark \ref{rmrk: locally closed shadow is riso-strat} that $X_\mathrm{sm}\subset \mathcal{S}_2(X/k)$. Notice that $V=\{x=y=0\}=X_\mathrm{sing}\subset \A_k^3$, and $X(K)$ is at most $V$-riso-trivial on $B_{X,x_0}$ for a singular point $x_0\in X_\mathrm{sing}$ according to Proposition \ref{Prop: invariant image in the cone}. Now, let $\pi_V\colon B_{X,x_0}\to V(\O_K)$ be the projection that forgets the $x,y$-coordinates. Let $x_0=(0,0,z_0)\in X_\mathrm{sing}(k)\setminus \{p\}$, notice that we have an embedded set-theoretic description:
\begin{align}\label{eq: decomp-example2}
    \pi_V^{-1}(v)= B_{S_{v},v} \cup B_{H_{v},v}
\end{align}
for every $v\in B_{V,x_0}$, where $S_{v}=\{x+\omega y=0\}, ~H_{v}=\{x-\omega y=0\}\subset \A_k^3$ and $\omega\in R$ is a square root of $v$ (which exists since $K$ is algebraically closed, and $v\neq 0$). Moreover, if we set $\omega_0\in K$ a square root of $z_0$, then we can choose $\omega$ uniformly for each $v\in B_{V,x_0}$ in the sense of $S_{v}$ and $H_v$ corresponding respectively to $S_{x_0}$ and $H_{x_0}$ in terms of the rv-classes they define when taking differences to its image by $\pi_V$. From this, we construct the following function:
\begin{align*}
    \varphi\colon B_{X,x_0}&\longrightarrow \A_k^3(\O_K)\\
    \alpha=(\alpha_1,\alpha_2,\alpha_3)&\longmapsto \begin{cases}
        (-\alpha_1\omega_0,\alpha_2,\alpha_3)~,~\text{if $\alpha\in B_{S_{\alpha_3},\alpha_3}$}\\
        (\alpha_1\omega_0,\alpha_2,\alpha_3)~,~\text{if $\alpha\in B_{H_{\alpha_3},\alpha_3}$}.
    \end{cases}
\end{align*}
It is an exercise to check that this defines a risometry to its image $\pi^{-1}_V(x_0)+B_{V,x_0}$. This confirms $X_\mathrm{sing}\setminus\{p\}\subset \mathcal{S}_1(X/k)$. Finally, this same decomposition (\ref{eq: decomp-example2}) can be used to measure the fibers of $\pi_V\colon B_{X,p}\to V(\O_K)$, and check that the central fiber has strictly less rv-classes of differences than any other fiber, hence Proposition \ref{Prop: Improved riso} cannot hold. Therefore, $\{p\}\subset \mathcal{S}_0(X/k)$ which concludes that the first shadow has locally closed strata, hence it already computes the riso-stratification.
\end{proof}

The existing examples in algebraic situations suggest that the iterated shadow procedure, defined in \ref{def: riso-strat}, stabilizes at the first application for affine $k$-schemes. This is proved to be false for strictly semi-algebraic sets \cite{BWH25}*{4.2.4}. However, in our more restrictive algebraic situation, the geometric interpretation of a shadow not-stabilizing seems difficult to be achieved and so the situation is not so clear. In particular, the following question would be helpful for clarifying this situation:
\begin{Question}\label{Question 1}
    Is $\mathrm{rtd}_x(X;\O_K)\geq r$ an open condition on $X$?  
\end{Question}

We see from the previous examples that the first shadow construction, which always makes sense as a partition of $k$-points via our general definitions, does return constructible strata defining a reasonable stratification for simple singularities in char $p>0$. It is very natural to ask if this holds in more generality. The following question is one of the first that should be answered in this direction of testing riso-stratifications in char $p$:          
\begin{Question}\label{Question 2}
    Does Corollary \ref{Cor: rtdsp independent of K} hold in arbitrary characteristic?
\end{Question}


\bigskip
\end{document}